\documentclass[11pt]{amsart}
\usepackage{palatino, mathpazo}
\usepackage{amsfonts}
\usepackage[mathscr]{eucal}
\usepackage[all]{xy}
\usepackage{mathrsfs}

\newtheorem{theorem}{Theorem}[section]
\newtheorem{lemma}[theorem]{Lemma}
\newtheorem{proposition}[theorem]{Proposition}
\newtheorem{corollary}[theorem]{Corollary}

\theoremstyle{remark}
\newtheorem{remark}[theorem]{Remark}
\newtheorem{example}[theorem]{Example}
\numberwithin{equation}{section}

\newcommand{\T}{\mathscr{F}}
\newcommand{\R}{\mathscr{R}}
\newcommand{\Mbar}{\overline{M}}
\newcommand{\vir}{\operatorname{vir}}
\newcommand{\ev}{\operatorname{ev}}
\newcommand{\ft}{\operatorname{ft}}
\newcommand{\st}{\operatorname{st}}
\newcommand{\on}{\operatorname}
\newcommand{\s}{s}
\newcommand{\e}{\varepsilon}
\newcommand{\q}{q}

\newcommand{\bb}{\mathbb}

\newcommand{\scr}{\mathscr}

\newcommand{\N}{NE_f}
\newcommand{\f}{\mathscr{G}}

\title[Invariance of GW theory under a simple flop]
{Invariance of Gromov--Witten theory\\ under a simple flop}

\author{Y.~Iwao}
\author{Y.-P.~Lee}
\address{Y.~Iwao and Y.-P.~Lee: Department of Mathematics, University of Utah,
Salt Lake City, Utah 84112-0090, U.S.A.} 
\email{yshr@math.utah.edu}
\email{yplee@math.utah.edu}

\author{H.-W.~Lin}
\author{C.-L.~Wang}
\address{H.-W.~Lin and C.-L.~Wang: Department of Mathematics and
The NCU Center for Mathematics and Theoretic Physics,
National Central University,
Jhongli 32001, Taiwan; 
National Center for Theoretical Sciences (NCTS),
Hsinchu 30013, Taiwan.} 
\email{linhw@math.ncu.edu.tw}
\email{dragon@math.cts.nthu.edu.tw; dragon@math.ncu.edu.tw}

\begin{document}
\maketitle

\begin{abstract}
We show that the generating functions of Gromov--Witten invariants
\emph{with ancestors} are invariant under a simple flop, for all
genera, after an analytic continuation in the extended K\"ahler
moduli space. This is a sequel to \cite{LLW}.
\end{abstract}

\setcounter{section}{-1}

\section{Introduction}

\subsection{Statement of the main results}

Let $X$ be a smooth complex projective manifold and $\psi:X \to
\bar X$ a flopping contraction in the sense of minimal model
theory, with $\bar\psi: Z \cong \mathbb{P}^r \to pt$ the
restriction map to the extremal contraction. Assume that $N_{Z/X}
\cong \mathscr{O}_{\mathbb{P}^r}(-1)^{\oplus(r + 1)}$. It was
shown in \cite{LLW} that a simple $\mathbb{P}^r$ flop $f: X
\dashrightarrow X'$ exists and the graph closure $[\bar\Gamma_f]
\in A^*(X \times X')$ induces a correspondence $\T$ which
identifies the Chow motives $\hat{X}$ of $X$ and $\hat{X'}$ of
$X'$. Furthermore, the big quantum cohomology rings, or
equivalently genus zero Gromov--Witten invariants with $3$ or more
insertions, are invariant under a simple flop, after an analytic
continuation in the extended K\"ahler moduli space.

The goal of the current paper is to extend the results of \cite{LLW} 
to all genera. In the process we discovered the natural framework in the
\emph{ancestor} potential
\begin{equation*}
 \mathscr{A}_X (\bar{t}, \s) := \exp {\sum_{g=0}^{\infty}
 \hbar^{g - 1} \overline{F}_g^X(\bar{t}, \s)},
\end{equation*}
which is a formal series in the Novikov variables
$\{q^\beta\}_{\beta \in NE(X)}$ defined in the stable range $2g +
n \ge 3$. See Section~\ref{s:1} for the definitions.

The main results of this paper are the following theorems.

\begin{theorem} \label{t:main}
The total ancestor potential $\mathscr{A}_X$
(resp.~$\mathscr{A}_{X'}$) is analytic in the extremal ray
variable $q^\ell$ (resp.~$q^{\ell'}$). They are identified via
$\T$ under a simple flop, after an analytic continuation in the
extended K\"ahler moduli space $\omega \in H^{1,1}_{\mathbb{R}}(X)
+ i(\mathcal{K}_X \cup \T^{-1}\mathcal{K}_{X'})$ via
$$
q^\ell = e^{2\pi i (\omega.\ell)},
$$
where $\mathcal{K}_X$ (resp.~$\mathcal{K}_{X'}$) is the K\"ahler
cone of $X$ (resp.~$X'$).
\end{theorem}

There are extensive discussions of analytic continuation and 
the K\"ahler moduli in Section~\ref{s:3}.
We note that the \emph{descendent} potential is in general
\emph{not} invariant under $\T$ (c.f.~\cite{LLW}, \S3). The
descendents and ancestors are related via a simple transformation
(\cite{KM, aG2}, c.f.~Proposition \ref{p:1.1}), but the
transformation is in general not compatible with $\T$.
Nevertheless we do have

\begin{theorem} \label{t:main-2}
For a simple flop $f$, any generating function of mixed invariant of
$f$-special type
\[
\langle \tau_{k_1,\bar{l}_1} \alpha_1, \cdots,
\tau_{k_n,\bar{l}_n} \alpha_n \rangle_{g},
\]
with $2g + n \ge 3$, is invariant under $\T$ up to analytic
continuation.
\end{theorem}

Here a mixed insertion $\tau_{k,\bar{l}\,}\alpha$ consists of
descendents $\psi^k$ and ancestors $\bar{\psi}^l$. Given $f: X
\dasharrow X'$ with exceptional loci $Z \subset X$ and $Z' \subset
X'$, a mixed invariant is of $f$-special type if for every
insertion $\tau_{k, \bar{l}\,} \alpha$ with $ k \ge 1$ we have $\alpha.Z = 0$. 
The generating function is a summation of all degrees and number of marked
points. See section 1 for the definitions. 
Theorem~\ref{t:main} is a special case of Theorem \ref{t:main-2} when no
descendent is present.

\subsection{Outline of the contents}

Section~1 contains some basic definitions as well as special
terminologies in Gromov--Witten theory used in the article. One of
the main ingredients of our proof of invariance in the higher
genus theory is Givental's quantization formalism
\cite{aG2} for \emph{semisimple} Frobenius manifolds. This is
reviewed in Section~2.

Another main ingredient, in comparing Gromov--Witten theory of $X$
and $X'$, is the degeneration analysis. 
We generalize the genus zero results of the degeneration analysis in 
\cite{LLW} to ancestor potentials in all genera. 
The analysis allows us to reduce the proofs of Theorem \ref{t:main} 
(and \ref{t:main-2}) from flops of $X$ to flops of
the local model $\mathbb{P}(N_{Z/X} \oplus \mathscr{O})$.

To keep the main idea clear, we choose to work on local models
first in Section~3 and postpone the degeneration analysis till
section~4. The local models are semi-Fano toric varieties and
localizations had been effectively used to solve the genus zero
case. The idea is to utilize Givental's quantization formalism on
the local models to derive the invariance in higher genus, up to
analytic continuation, from our results \cite{LLW} in genus zero.
In doing so, the key point is that local models have
semisimple quantum cohomology, and we trace the effect of
analytic continuation carefully during the process of quantization.
The issues of the analyticity of the Frobenius manifolds and
the analytic continuation involved in this study is discussed in 
the beginning Section~3.

The proofs of our main results Theorem~\ref{t:main} and
\ref{t:main-2}, as well as the degeneration analysis, are
presented in section~4.

In section~5 we include some discussions and calculations of the higher
genus Gromov--Witten invariants attached to the extremal rays.
Similar to the $g = 0$ case, there is also a classical defect
occurring at $(g, n,d) = (1,1,0)$
$$
-\frac{1}{24}\big[(c_{{\rm top} - 1}(X).\alpha)_X - (c_{{\rm top}
- 1}(X').\T\alpha)_{X'} \big].
$$
Our explicit formula in Theorem~\ref{t:g1} 
for the $g = 1$ invariants attached to the
extremal ray is seen to give quantum corrections to it.

The calculation of the explicit formula in genus one requires some
elementary combinatorics, and is included in the appendix.

\subsection{Some remarks on the crepant resolution conjecture}

A morphism $\psi:X \to \bar{X}$ is called a \emph{crepant
resolution}, if $X$ is smooth and $\bar{X}$ is
$\mathbb{Q}$-Gorenstein (e.g.~an orbifold) such that $\psi^*
K_{\bar{X}} = K_X$. In the case $\bar{X}$ is an orbifold, there is
a well-defined orbifold Gromov--Witten theory due to Chen--Ruan.
The \emph{crepant resolution conjecture} asserts a close relation
between the Gromov--Witten theory of $X$ and that of $\bar{X}$.

Crepant resolution conjecture, as formulated in \cite{CR}, still
uses descendent potentials rather than the ancestor potentials, as
advocated in \cite{ypL}. Yet ancestors often enjoy better
properties than the corresponding descendents, as exploited by
Getzler \cite{eG}.

Since different crepant resolutions are related by a
$K$-equivalent transformation, e.g.~a flop, the conjecture must be
consistent with a transformation under a flop. Although the
descendent potentials can be obtained from ancestor potentials via
a simple transformation, this very transformation actually spoils
the invariance under $\T$. The insistence in the descendents may
introduce unnecessary complication in the formulation of
the conjecture. This is especially relevant in the stronger form
of the conjecture when the orbifolds satisfy the Hard Lefschetz
conditions.

Our result suggests that a more natural framework to study crepant
resolution conjecture is to use ancestors rather than descendents.
We leave the interested reader to consult \cite{CR} and references
therein.

\subsection{Acknowledgements}

Part of this work was done during the second author's visit to the
NCU Center for Mathematics and Theoretic Physics (CMTP), Jhongli,
Taiwan in November 2007. He is grateful to the Mathematics
Department of National Central University for the hospitality
during his stay.

\section{Descendent and ancestor potentials} \label{s:1}

\subsection{The ancestor potential}

For the stable range $2g + m \ge 3$, let
\[
 \pi: = \ft \circ \st: \Mbar_{g,m+l}(X,\beta) \to \Mbar_{g, m}
\]
be the composition of the \emph{stabilization morphism}
$\st: \Mbar_{g,m+l}(X,\beta) \to \Mbar_{g,m+l}$
defined by forgetting the map and the \emph{forgetful morphism}
$\ft: \Mbar_{g,m+l} \to \Mbar_{g,m}$
defined by forgetting the last $l$ points.
The \emph{gravitational ancestors} are defined to be
\begin{equation} \label{e:pi}
  \bar{\psi}_i := \pi^* \psi_i
\end{equation}
for $i=1,\ldots,m$.

Let $\{T_\mu\}$ be a basis of $H^*(X, \mathbb{Q})$. Denote
$\bar{t} = \sum_{\mu, k} \bar{t}_k^{\mu} \bar{\psi}^k T_{\mu}$,
$\s = \sum_{\mu} \s^{\mu} T_{\mu}$, and let
\[
 \begin{split}
   \overline{F}^X_g (\bar{t}, \s) = &\sum_{m,l,\beta}
   \frac{q^{\beta}}{m! l!}
   \langle {\bar{t}}^m, \s^l \rangle_{g, m+l, \beta} \\
   = &\sum_{m,l,\beta} \frac{q^{\beta}}{m! l!}
     \int_{[\Mbar_{g,m+l}(X,\beta)]^{\vir}}
     \prod_{i=1}^m \sum_{k, \mu} \bar{t}_{k}^{\mu} \bar{\psi}_i^k
     \ev_i^* T_{\mu}
     \prod_{i=m+1}^{m+l} \sum_{\mu} \s^{\mu} \ev_i^* T_{\mu}
 \end{split}
\]
be the generating function of genus $g$ ancestor invariants.
The ancestor potential is defined to be the formal expression
\begin{equation*}
 \mathscr{A}_X (\bar{t}, \s) := \exp {\sum_{g=0}^{\infty}
 \hbar^{g - 1} \overline{F}_g^X(\bar{t}, \s)}.
\end{equation*}
Note that $\mathscr{A}$ depends on $\s$ (variables on the
Frobenius manifold), in addition to $\bar{t} = \sum
\bar{t}_k^{\mu} T_{\mu} z^k$ (variables on the ``Fock space'').

Let $j$ be one of the first $m$ marked points such that
$\bar{\psi}_j$ is defined. Define $D_j$ be the (virtual) divisor
on $\Mbar_{g,m+l}(X,\beta)$ defined by the image of the gluing
morphism
\[
 \sum_{\beta'+\beta''=\beta} \sum_{l' + l''=l}
 \Mbar_{0, \{ j \}+l'+ \bullet}(X,\beta') \times_X
  \Mbar_{g,(m-1)+l''+ \bullet}(X,\beta'') \to \Mbar_{g,m+l}(X,\beta),
\]
where $\bullet$ represents the gluing point;
$\Mbar_{g,(m-1)+l''+\bullet}(X,\beta'')$ carries all first $m$
marked points \emph{except the $j$-th one}, which is carried by
$\Mbar_{0, \{ j \}+l'+\bullet}(X,\beta')$. Ancestor and descendent
invariants are related by the simple geometric equation
\begin{equation} \label{e:compare}
(\psi_j - \bar{\psi}_j) \cap [\Mbar_{g,m+l}(X,\beta)]^{\vir} =
[D_j]^{\vir}.
\end{equation}
This can be easily seen from the definitions of $\psi$ and $\bar{\psi}$.
The morphism $\pi$ in \eqref{e:pi} contracts only rational curves
during the processes of forgetful and stabilization morphisms. The
(virtual) difference of $\psi$ and $\bar{\psi}$ is exactly $D_j$.

\subsection{The mixed invariants}

We will consider more general \emph{mixed invariants} with mixed
ancestor and descendent insertions. Denote by
\[
\langle \tau_{k_1,\bar{l}_1} \alpha_1, \cdots,
\tau_{k_n,\bar{l}_n} \alpha_n \rangle_{g,n,\beta}
\]
the invariants with mixed descendent and ancestor insertion
$\psi_i^{k_i} \bar{\psi}_i^{l_i} \ev_i^* \alpha_i$ at the $i$-th
marked point and let
\[
  \langle \tau_{k_1,\bar{l}_1} \alpha_1, \cdots,
    \tau_{k_n,\bar{l}_n} \alpha_n \rangle_g (\s)
  := \sum_{l, \beta} \frac{q^{\beta}}{l!}
  \langle \tau_{k_1,\bar{l}_1} \alpha_1 , \cdots,
 \tau_{k_n,\bar{l}_n} \alpha_n, \s^l \rangle_{g,n+l,\beta},
\]
$$
\langle \tau_{k_1,\bar{l}_1} \alpha_1, \cdots,
\tau_{k_n,\bar{l}_n} \alpha_n \rangle_g (\bar{t}, \s) := \sum_{m,
l, \beta} \frac{q^{\beta}}{m! l!} \langle \tau_{k_1,\bar{l}_1}
\alpha_1 , \cdots, \tau_{k_n,\bar{l}_n} \alpha_n, \bar{t}^m, \s^l
\rangle_{g,n+m+l,\beta} .
$$
to be the generating functions.

Equation~\eqref{e:compare} can be rephrased in terms of these
generating functions.

\begin{proposition} \label{p:1.1}
In the stable range $2g + n \ge 3$,
\begin{equation} \label{e:1.3}
\begin{split}
&\langle \tau_{k+1,\bar{l}\,} \alpha_{1}, \cdots \rangle_g (\bar{t},
\s) \\
&= \langle \tau_{k, \overline{l+1}\,} \alpha_{1}, \cdots \rangle_g
(\bar{t}, \s) + \sum_{\nu} \langle \tau_{k\,} \alpha_{1}, T_{\nu}
\rangle_0(\s)\, \langle \tau_{\overline{l}\,} T^{\nu}, \cdots
\rangle_g(\bar{t}, \s)
\end{split}
\end{equation}
where $\cdots$
denote the same mixed insertions.
\end{proposition}

In fact, only one special type of the mixed invariants will be needed.
Let $(X, E)$ be a smooth pair with $j: E \hookrightarrow X$ a smooth
(infinity) divisor.
At the $i$-th marked point, if $k_i \ne 0$, then we require that
$\alpha_i = \e_i \in j_* H^*(E) \subset H^*(X)$. This type of
invariants will be called \emph{mixed invariants of special type}
and the marked points with $k_i \ne 0$ will be called
\emph{marked points at infinity}.

For a birational map $f: X \dasharrow X'$ with exceptional
loci $Z \subset X$, a mixed invariant is said to be of $f$-special
type if $\alpha.Z=0$ for every insertion $\tau_{k, \bar{l}\,} \alpha$
with $ k \neq 0$. When $(X_{loc}, E)$ comes form the
local model of $(X, Z)$, namely $X_{loc} := \tilde E =
\mathbb{P}_Z(N_{Z/X} \oplus \mathscr{O})$ with $E$ being the
infinity divisor, these two notions of special type agree.

Proposition \ref{p:1.1} will later be used (c.f.~Theorem
\ref{t:local}) in the following setting. Suppose that under a flop
$f: X \dasharrow X'$ we have invariance of ancestor generating
functions. To extend the invariance to allow also descendents we
may reduce the problem to the $g = 0$ case and with at most one
descendent insertion $\tau_{k\,} \alpha$.
For local models, it is important that the invariants are of special type
to ensure the invariance.

\section{Review of Givental's quantization formalism}

\subsection{Formal ingredients in the geometric Gromov--Witten theory}
For a projective smooth variety $X$, Gromov--Witten theory of $X$
consists of the following ingredients
\begin{enumerate}
\item[(i)] $H := H^*(X, \mathbb{Q})$ is a $\mathbb{Q}$-vector
space, assumed of rank $N$. Let $\{ T_{\mu} \}_{\mu=1}^N$ be a
basis of $H$ and $\{ \s^{\mu} \}_{\mu=1}^N$ be the dual
coordinates with ${\partial}/{\partial \s^{\mu}} = T_{\mu}$.
$\mathbf{1} \in H^0(X)$, the (dual of) fundamental class, is a
special element. $H$ carries a symmetric bilinear form, Poincar\'e
pairing,
\[
  ( \cdot, \cdot ) : H \otimes H \to \mathbb{Q}.
\]
Define
\[
  g_{\mu \nu} := ( T_{\mu}, T_{\nu})
\]
and $g^{\mu \nu}$ to be the inverse matrix.

\item[(ii)] Let $\mathscr{H}_t := \oplus_{k=0}^{\infty} H$ be the
infinite dimensional complex vector space with basis $\{ T_{\mu}
\psi^k \}$. $\mathscr{H}_t$ has a natural $\mathbb{Q}$-algebra
structure:
\[
 T_{\mu} \psi^{k_1} \otimes T_{\nu} \psi^{k_2} \mapsto
 (T_{\mu} \cup T_{\nu})  \psi^{k_1 + k_2}.
\]
Let $\{ t^{\mu}_k \}$, $\mu=1, \ldots, N$, $k=0, \ldots, \infty$,
be the dual coordinates of the basis $\{ T_{\mu} \psi^k \}$. We
note that at each marked point, the insertion is
$\mathscr{H}_t$-valued. Let
\[
 t:= \sum_{k, \mu} t^{\mu}_k T_{\mu} \psi^k
\]
denote a general element in the vector space $\mathscr{H}_t$.

\item[(iii)] The generating function of descendents
$F^X_g(t)$ is a formal function on $\mathscr{H}_t$. The generating
function of ancestors $\bar{F}^X_g(\bar{t})$ is a formal function
of $(\s,\bar{t}) \in H \times \mathscr{H}_t$.

\item[(iv)] $H$ carries a (big quantum cohomology) ring structure.
Let $\s^{\mu} = t^{\mu}_0$ and $F_0(\s) = F_0 (t)|_{t_k = 0, \
\forall\, k > 0}$. The ring structure is defined by
\[
T_{\mu_1} *_{\s} T_{\mu_2} := \sum_{\nu,\nu'} \frac{\partial^3
F_0(\s)}{\partial \s^{\mu_1} \partial \s^{\mu_2}
\partial \s^{\nu}}  g^{\nu \nu'} T_{\nu} .
\]
$\mathbf{1}$ is the identity element of the ring.

\item[(v)] The Dubrovin connection $\nabla_z$ on the tangent bundle $TH$
is defined by
\[
\nabla_z := d - z^{-1} \sum_{\mu} d \s^{\mu} (T_{\mu} *).
\]
The quantum cohomology differential equation
\begin{equation} \label{e:qde}
\nabla_z S =0
\end{equation}
has a fundamental solution $J = (J_{\mu, \nu} (\s, z^{-1}))$, an
$N \times N$ matrix-valued function, in (formal) power series of
$z^{-1}$ satisfying the conditions
\begin{equation} \label{e:sympl}
J(\s, z^{-1}) = Id + O(z^{-1})\ \text{and} \ J^*(\s, -z^{-1})
J(\s, z^{-1})= Id,
\end{equation}
where ${}^*$ denotes the adjoint with respect to $(\cdot.\cdot)$.

\item[(vi)] The non-equivariant genus zero Gromov--Witten theory
is graded, i.e.~with a \emph{conformal} structure. The grading is
determined by an \emph{Euler field} $E \in \Gamma(T_X)$,
\begin{equation} \label{Euler}
 E = \sum_{\mu} (1 - \frac{1}{2} \deg T_{\mu} ) \s^{\mu}
 \frac{\partial}{\partial \s^{\mu}} + c_1(T_X).
\end{equation}
\end{enumerate}

\subsection{Semisimple Frobenius manifolds}

The concept of Frobenius manifolds was originally introduced by
B.~Dubrovin. We assume that the readers are familiar with the
definitions of the Frobenius manifolds. See \cite{LP2} Part I for
an introduction. The quantum product $*$, together with Poincar\'e
pairing, and the special element $\mathbf{1}$, defines on $H$ a
Frobenius manifold structure $(QH, *)$.

A point $\s \in H$ is called a \emph{semisimple} point if the
quantum product at the tangent algebra $(T_{\s}H, *_{\s})$ at $\s
\in H$ is isomorphic to $\oplus_1^N \mathbb{C}$ as an algebra.
$(QH,*)$ is called semisimple if the semisimple points is dense in
$H$. If $(QH, *)$ is semisimple, it has idempotents $\{ \epsilon_i
\}_1^N$
\[
  \epsilon_i * \epsilon_j = \delta_{ij} \epsilon_i .
\]
defined up to $S_N$ permutations. The \emph{canonical coordinates}
$\{ u^i \}_1^N$ are defined by ${\partial}/{\partial u^i} =
\epsilon_i$. When the Euler field is present, the canonical
coordinates are also uniquely defined up to permutations. We will
often use the \emph{normalized} form
\[
\tilde{\epsilon_i} = \frac{1}{\sqrt{(\epsilon_i, \epsilon_i)}}\,
\epsilon_i.
\]

\begin{lemma}
$\{ \epsilon_i \}$ and $\{ \tilde{\epsilon}_i \}$ form orthogonal bases.
\end{lemma}

\begin{proof}
\begin{align*}
(\epsilon_i, \epsilon_j) &= (\epsilon_i * \epsilon_i, \epsilon_j)
= (\epsilon_i, \epsilon_i * \epsilon_j)\\
&= (\epsilon_i, \delta_{ij} \epsilon_i ) = \delta_{ij}
(\epsilon_i, \epsilon_i).
\end{align*}
\end{proof}

When the quantum cohomology is semisimple,
the quantum differential equation \eqref{e:qde} has
a fundamental solution of the following type
\[
 \mathbf{R} (\s,z):= \Psi(\s)^{-1} R(\s, z) e^{\mathbf{u}/z} ,
\]
where $(\Psi_{\mu i}) := (T_{\mu}, \tilde{\epsilon_i} )$ is the
transition matrix from $\{ \tilde{\epsilon_i} \}$ to $\{ T_{\mu}
\}$; $\mathbf{u}$ is the diagonal matrix $(\mathbf{u}_{ij}) =
\delta_{ij} u^i$. The main information of $\mathbf{R}$ is carried
by $R(\s,z)$, which is a (formal) power series in $z$. One notable
difference between $J(\s, z^{-1})$ and $R(\s,z)$ is that the
former is a (formal) power series in $z^{-1}$ while the latter is
a (formal) power series in $z$. 
See \cite{aG2} and Theorem~1 in Chapter~1 of \cite{LP2}.

\subsection{Preliminaries on quantization}

Let $\mathscr{H}_q := H[z]$. Let $\{ T_{\mu} z^k
\}_{k=0}^{\infty}$ be a basis of $\mathscr{H}_q$, and $\{
q^{\mu}_k \}$ the dual coordinates. We define an isomorphism of
$\mathscr{H}_q$ to $\mathscr{H}_t$ as an affine vector space via a
\emph{dilaton shift} ``$t = q + z$'':
\begin{equation} \label{shift}
   t^{\mu}_k = q^{\mu}_k + \delta^{\mu \mathbf{1}} \delta_{k 1}.
\end{equation}
The cotangent bundle $\mathscr{H} := T^* \mathscr{H}_q$ has a
natural symplectic structure
$$
\Omega = \sum_{k, \mu, \nu} g_{\mu \nu}\, dp^\mu_k \wedge dq^\nu_k
$$
where $\{ p^{\mu}_k \}$ are the dual coordinates in the fiber
direction of $\mathscr{H}$ in the natural basis $\{ T_{\mu}
z^{-k-1} \}_{k=0}^{\infty}$. $\mathscr{H}$ is naturally isomorphic
to the $H$-valued Laurent series in $z^{-1}$, $H[\![z^{-1}]\!]$.
In this way, then
$$
\Omega(f, g) = {\rm Res}_{z = 0} (f(-z), g(z)).
$$

To quantize an infinitesimal symplectic transformation on
$(\mathscr{H}, \Omega)$, or its corresponding quadratic
hamiltonians, we recall the standard Weyl quantization. An
identification $\mathscr{H}=T^* \mathscr{H}_q$ of the symplectic
vector space $\mathscr{H}$ (the \emph{phase space}) as a cotangent
bundle of $\mathscr{H}_q$ (the \emph{configuration space}) is
called a polarization. The ``Fock space'' will be a certain class
of functions $f(\hbar, q)$ on $\mathscr{H}_q$ (containing at least
polynomial functions), with additional formal variable $\hbar$
(``Planck's constant''). The classical observables are certain
functions of $p, q$. The quantization process is to find for the
classical mechanical system on $(\mathscr{H}, \Omega)$ a
``quantum'' system on the Fock space such that the classical
observables, like the hamiltonians $h(q,p)$ on $\mathscr{H}$, are
quantized to become operators $\widehat{h}(q, {\partial}/{\partial
q})$ on the Fock space.

Let $A(z)$ be an $\on{End}(H)$-valued Laurent formal series in $z$
satisfying
\[
  \Omega(A f, g) + \Omega(f, A g)=0,
\]
for all $f,g \in \mathscr{H}$. That is, $A(z)$ defines an
infinitesimal symplectic transformation. $A(z)$ corresponds to a
quadratic polynomial \footnote{Due to the nature of the infinite
dimensional vector spaces involved, the ``polynomials'' here might
have infinite many terms, but the degrees remain finite.} $P(A)$
in $p, q$
\[
  P(A)(f) := \frac{1}{2} \Omega(Af, f) .
\]

Choose a \emph{Darboux coordinate system} $\{ q^i_k, p^i_k \}$ so
that $\Omega = \sum dp^i_k \wedge dq^i_k$. The quantization $P
\mapsto \widehat{P}$ assigns
\begin{equation} \label{e:wq}
\begin{split}
&\widehat{1}= 1, \  \widehat{p_k^{i}}= \sqrt{\hbar}
\frac{\partial}{\partial q^{i}_k}, \
\widehat{q^{i}_k} = q^{i}_k / {\sqrt{\hbar}}, \\
&\widehat{p^{i}_k p^{j}_l} = \widehat{p^{i}_k} \widehat{p^{j}_l}
=\hbar \frac{\partial}{\partial q^{i}_k}
\frac{\partial}{\partial q^{j}_l}, \\
&\widehat{p^{i}_k q^{j}_l} = q^{j}_l \frac{\partial}{\partial q^{i}_k},\\
&\widehat{q^{i}_k q^{j}_l} = {q}^{i}_k {q}^{j}_l /\hbar ,
\end{split}
\end{equation}
In summary, the quantization is the process
\[
\begin{matrix}
  A   &\mapsto & P(A)  &\mapsto & \widehat{P(A)} \\
  \text{inf. sympl. transf.}  &\mapsto & \text{quadr. hamilt.}
    &\mapsto & \text{operator on Fock sp.}.
\end{matrix}
\]
It can be readily checked that the first map is a Lie algebra
isomorphism: The Lie bracket on the left is defined by $[A_1,
A_2]=A_1 A_2 - A_2 A_1$ and the Lie bracket in the middle is
defined by Poisson bracket
\[
\{ P_1(p,q), P_2(p,q) \} = \sum_{k,i} \frac{\partial P_1}{\partial
p^{i}_k} \frac{\partial P_2}{\partial q^{i}_k} -\frac{\partial
P_2}{\partial p^{i}_k} \frac{\partial P_1}{\partial q^{i}_k}.
\]
The second map is close to be a Lie algebra homomorphism. Indeed
\[
 [\widehat{P_1},\widehat{P_2}] =  \widehat{\{ P_1, P_2 \}} +
  \mathscr{C}(P_1,P_2),
\]
where the cocycle $\mathscr{C}$, in orthonormal coordinates,
vanishes except
\[
 \mathscr{C}(p^{i}_k p^{j}_l, q^{i}_k q^{j}_l) =
 -\mathscr{C}(q^{i}_k q^{j}_l, p^{i}_k p^{j}_l)
 = 1 + \delta^{i j} \delta_{kl}.
\]

\begin{example} \label{e:2.2}
Let $\dim H=1$ and $A(z)$ be multiplication by $z^{-1}$.
It is easy to see that $A(z)$ is infinitesimally symplectic.
\begin{equation} \label{e:cse}
 \begin{split}
 P(z^{-1})= &-\frac{q_0^2}{2} - \sum_{m=0}^{\infty} q_{m+1} p_m \\
 \widehat{P(z^{-1})} = &-\frac{q_0^2}{2}
            - \sum_{m=0}^{\infty} q_{m+1} \frac{\partial}{\partial q_m}.
 \end{split}
\end{equation}
\end{example}

Note that one often has to quantize the symplectic instead of the
infinitesimal symplectic transformations.
Following the common practice in physics, define
\begin{equation} \label{e:q}
  \widehat{e^{A(z)}} := e^{\widehat{A(z)}},
\end{equation}
for $A(z)$ an infinitesimal symplectic transformation.

\subsection{Ancestor potentials via quantization}

Let $N$ be the rank of $H=H^*(X)$ and $\mathscr{D}_{N}({\bf t}) =
\prod_{i = 1}^N \mathscr{D}_{pt}(t^i)$ be the descendent potential
of $N$ points, where
\begin{equation*}
\mathscr{D}_{pt}(t^i) \equiv \mathscr{A}_{pt} (t^i) := \exp
{\sum_{g=0}^{\infty}
 \hbar^{g - 1} F^{pt}_g(t^i)}
\end{equation*}
is the total descendent potential on a point 
and $t^i = \sum_k t^i_k z^k$.

Suppose that $(QH, *)$ is semisimple, then the ancestor potential
can be reconstructed from the $\mathscr{D}_{N}({\bf t})$ via the
the quantization formalism.

First of all, $\{ \tilde{\epsilon}_i \}$ define an
\emph{orthonormal basis} for $H$ with canonical coordinates
$\{u^i\}_{i = 1}^N$. Therefore, the dual coordinates $(p^i_k,
q^i_k)$ of the basis $\{ \tilde{\epsilon}_i z^k \}_{k \in
\mathbb{Z}}$ for $\mathscr{H}$ form a Darboux coordinate system.
The coordinate system ${\bf t} = \{t^i_k\}$ is related to ${\bf q}
= \{q^i_k\}$ by the dilaton shift \eqref{shift}.
Note that $\partial/\partial q^i_k = \partial/\partial t^i_k$.

The following beautiful formula was first formulated by Givental \cite{aG2}.
Many special cases have since been solved by Givental and others
\cite{BCK}, \cite{ypL2}.
It was completely established by C.~Teleman in a recent preprint \cite{cT}.

\begin{theorem}[\cite{aG2, cT}] \label{quantization}
\begin{equation} \label{e:ancestor}
\mathscr{A}_X (\bar t, s) = e^{\bar{c}(\s)}
\widehat{\Psi}^{-1}(\s) \widehat{R}_X(s, z)
e^{\widehat{\mathbf{u}/z}}(\s) \mathscr{D}_{N}({\bf t}),
\end{equation}
where $\bar{c}(\s) = \frac{1}{48} \ln \det
(\epsilon_i, \epsilon_j)$.
\end{theorem}

Note that it is not very difficult to check that $\ln {R}_X(s, z)$
defines an infinitesimal symplectic transformation. See e.g.~\cite{aG2, LP2}.
$\widehat{R}_X(s, z)$ is therefore well-defined.
By Example~\ref{e:2.2}, $e^{\widehat{\mathbf{u}/z}}$ is also well-defined.
Since the quantization involves only the $z$ variable,
$\widehat{\Psi}^{-1}(\s)$ really is the induced transformation
from canonical coordinates to flat coordinates. 
No quantization is needed.

\begin{remark}
The operator $e^{\widehat{\mathbf{u}/z}}$ can be removed from
the above expression. It is shown in \cite{aG2} that the string
equation implies that $e^{\widehat{\mathbf{u}/z}}\mathscr{D}_{N} =
\mathscr{D}_{N}$.
\end{remark}

\section{Analytic Continuation and Local models} \label{s:3}

In the first part of this section, we discuss the issues of the 
analyticity of the Frobenius manifolds and the analytic continuation 
involved in the study of the flops $f: X \dashrightarrow X'$.
We then move to the study of the local models. 
There the semisimplicity of the Frobenius manifolds and the quantization
formalism are used to reduce the invariance of Gromov--Witten theory
to the semi-classical (genus zero) case.

\subsection{Review of the genus zero theory}

Let $f:X \dasharrow X'$ be a simple $\mathbb{P}^r$ flop with $\T$
being the graph correspondence. This subsection rephrases the
analytic continuation of big quantum rings proved in \cite{LLW} in
more algebraic terms.

Let $\N$ be the cone of curve classes $\beta \in NE(X)$ with $\T
\beta \in NE(X')$, i.e.~the classes which are effective on both
sides. Let
$$
\f(q) = \frac{q}{1 - (-1)^{r + 1}q}
$$
be the rational function coming from the generating function of
three points Gromov--Witten invariants attached to the extremal ray
$\ell \subset Z \cong \mathbb{P}^r$ with positive degrees. Namely
for any $i, j, k \in \mathbb{N}$ with $i + j + k = 2r +1$,
$$
\f(q^\ell)= \sum_{d \ge 1} \langle h^i, h^j, h^k \rangle_{0, 3,
d\ell}\, q^{d\ell},
$$
where $h$ denotes a class in $X$ which restricts to the hyperplane
class of $Z$.

Define the ring
\begin{equation}
\R = \widehat{\mathbb{C}[\N]}[\f],
\end{equation}
which can be regarded as certain algebraization of the Novikov
ring $\widehat{NE(X)}$ in the $q^\ell$ variable. Notice that $\R$
is canonically identified with its counterpart $\R' =
\widehat{\mathbb{C}[\N']}[\f']$ under $\T$ since $\T \N = \N'$ and
\begin{equation} \label{f-eq}
\T \f(q^\ell) = (-1)^r - \f(q^{\ell'})
\end{equation}
(via $\f(q) + \f(q^{-1}) = (-1)^r$).

\begin{theorem} \label{QH-inv} 
The genus zero $n$-point functions with $n \ge 3$ lie in $\R$:
$$
\langle \alpha \rangle^X \in \R
$$
for all $\alpha \in H^*(X)^{\oplus n}$. Moreover $\T \langle
\alpha \rangle^X = \langle \T \alpha \rangle^{X'}$ in $\R'$.
\end{theorem}

\begin{proof}
This is the main result of \cite{LLW} except the statement that
$\langle \alpha \rangle^X \in \R$.
For this, the degeneration analysis in \S~4 of \cite{LLW} implies
$$
\langle \alpha \rangle^{\bullet X} = \sum_\mu m(\mu) \sum_I
\langle \alpha_1 \mid \e_I, \mu \rangle^{\bullet (Y, E)}\langle
\alpha_2 \mid \e^I, \mu \rangle^{\bullet (\tilde E, E)}.
$$
(A generalization to all genera is presented in the next section.)
Here $Y = {\rm Bl}_Z X = \bar\Gamma_f \subset X \times X'$;
$\langle \cdot \rangle^{\bullet}$ denote invariants with
possibly disconnected domain curves.
Under the projections $\phi: Y
\to X$ and $\phi': Y \to X'$, the variable $q^{\beta_1}$ for
$\beta_1 \in NE(Y)$ is identified with $q^{\phi_* \beta_1} \in
NE(X)$. If $q^{\beta_1}$ appears in the sum of contact type $\mu$,
then $(E.\beta_1) = |\mu| \ge 0$. Also
$$
\T \phi_* \beta_1 = \phi'_* \beta_1 + |\mu|\ell' \in NE(X').
$$
Hence $\langle \alpha_1 \mid \e_I, \mu \rangle^{\bullet (Y, E)}
\in \widehat{\mathbb{C}[\N]}$.

For the local model $\tilde E := \mathbb{P}_Z(N_{Z/X} \oplus
\mathscr{O})$, the process in \cite{LLW} (\S5, Theorem 5.6) via
reconstruction theorem shows that there are indeed only two
generators of the functional equations. One is (\ref{f-eq}), which
is the source of analytic continuation. Another one is the
quasi-linearity (\cite{LLW}, Lemma 5.4), which is an identity in
$\mathbb{C}[\N]$ where no analytic continuation is needed.

Denote
\begin{equation} \label{e:delta}
 \delta = \delta_h = q^{\ell} \frac{d}{dq^{\ell}}.
\end{equation}
Then all other analytic continuation come from $\delta^m \f$'s with
$m \ge 0$. It remains to show that $\delta^m \f$ is a polynomial in $\f$.
This follows easily from $\delta \f = \f + (-1)^{r + 1} \f^2$ and
$\delta(\f_1 \f_2) = (\delta \f_1) \f_2 + \f_1 \delta \f_2$ by
induction on $m$.
\end{proof}

\subsection{Integral structure on local models} \label{s:3.2}

For $X = \tilde E$, the above proof shows that
\begin{equation}
\langle \alpha \rangle \in \mathbb{C}[\N][\f] =: \R_{loc}
\end{equation}
without the need of taking completion, where $\N =
\mathbb{Z}_+\gamma + \mathbb{Z}_+(\gamma + \ell)$. In fact for a
given set of insertions $\alpha$ and genus $g$, the virtual
dimension count shows that the \emph{contact weight} $d_2 :=
(E.\beta)$ is fixed among all $\beta = d_1 \ell + d_2 \gamma$ in
the series $\langle \tau_{k, \bar l}\alpha \rangle_g^X$. Hence for
$g = 0$ we must have
\begin{equation*}
\langle \alpha \rangle^X = q^{d_2 \gamma} (p_0(\f) + q^{\ell}
p_1(\f) + \cdots + q^{d_2 \ell} p_{d_2}(\f))
\end{equation*}
for certain polynomials $p_i(\f) \in \mathbb{Z}[\f]$.

In particular $\langle \alpha \rangle$ is an analytic function
over the extended K\"ahler moduli $\omega \in
H^{1,1}_\mathbb{R}(X) + i(\mathcal{K}_X \cup
\T^{-1}\mathcal{K}_{X'})$ via the identification
\begin{equation} \label{e:3.5}
 q^\beta = e^{2\pi i (\omega.\beta)}.
\end{equation}
Thus analytic continuation can be taken in the traditional complex
analytic sense or as isomorphisms in the ring $\R_{loc} \cong
\R'_{loc}$.

\subsection{Analytic structure on the Frobenius manifolds}

The Frobenius manifold corresponding to $X$ is a priori a formal
scheme, given by the formal completing $\widehat{H}_X$ of
$H^*(X,\mathbb{C})$ at the origin. 
The big quantum product takes values in the \emph{Novikov ring}, 
or equivalently the \emph{formal} K\"ahler moduli 
$\widehat{\mathbb{C}[NE(X)]}$. 
The divisor axiom implies that one may combine $H^2(X, \mathbb{C})$
directions of the Frobenius manifold and the formal K\"ahler moduli
into a formal completion at $q=0$ of the complex torus
\begin{equation} \label{e:3.8}
 q \in \frac{H^2(X,\mathbb{C})}{H^2(X,\mathbb{Z})}.
\end{equation}
Indeed, let $s = s' + s_1$ be a point in the Frobenius manifold
with $s_1 \in H^2(X, \mathbb{C})$.
The divisor axiom says that
$$
\langle \alpha \rangle_\beta (s)\,q^\beta = \langle \alpha
\rangle_\beta (s')\, q^\beta e^{(s_1. \beta)}.
$$
Compared with \eqref{e:3.5}, this suggests an identification of the formal
K\"ahler moduli and the corresponding underlying Frobenius manifold
in the $H^2(X)$ direction into the complex torus \eqref{e:3.8}.
\footnote{In string theory, the identification of \emph{weights} 
$q^\beta = e^{2\pi i(\omega.\beta)}$ is essential in matching the $A$ model 
and $B$ model moduli spaces in mirror symmetry (cf.~\cite{CoKa}). 
It is generally believed that the GW theory converges in the ``large radius 
limit'', i.e.~when ${\rm Im}\,\omega$ is large.} 
Note that the ``origin'' of $s_1=0$ is moved to the origin of $q=0$ under
this identification.
In the case $\ell$ is the only primitive numerical class of curves and 
$h$ an ample class such that $(h.\ell) = 1$, one may set $s_1 = t h$.
Thus, we have the familiar identification
\begin{equation*}
q = q^\ell = e^t,
\end{equation*}
which will be used in the appendix. 
This identification can be done at the analytic
level in some directions of $H^2(X)$ when the convergence is known.

Let $f:X \dasharrow X'$ be a simple $\mathbb{P}^r$ flop and $h$ be
an ample divisor class dual to the extremal ray $\ell$,
i.e.~$(h.\ell) = 1$. Then $H^2(X, \mathbb{Z}) = \mathbb{Z}h \oplus
H^2(X, \mathbb{Z})^{\perp_\ell}$. Theorem~\ref{QH-inv} gives an
analytic structure on $\widehat{H}_X$ in the $h$-direction:
\begin{corollary} \label{c:3.2}

(i) The Frobenius manifold structure on $\widehat{H}_X$ can be
extended to
$$
 H_X := \widehat{H}_X^{\perp_\ell} \times
 (\mathbb{P}^1_{q^{\ell}} \setminus (-1)^{r+1}).
$$

(ii) $H_X \cong H_{X'}$.

(iii) If $X$ is the local model, $H_X$ is an analytic manifold.
\end{corollary}

\begin{proof}
Indeed, Theorem~\ref{QH-inv} says that, as functions of $\f$, all
invariants are defined on $\f \in \mathbb{C}$. Equivalently, as
function of $q^{\ell}$, all invariants are defined on
$\mathbb{P}^1 \setminus {(-1)^{r+1}}$. This proves (i). (ii)
follows from (i), and (iii) from Section~\ref{s:3.2}
\end{proof}

Corollary~\ref{c:3.2} and results of the previous subsections show
that the Frobenius manifold structures on the quantum cohomology
of $X$ and $X'$ are isomorphic. The former is a series expansion
of analytic functions at $q^{\ell} =0$, and the latter at
$q^{\ell} = \infty$. Considered as a one-parameter family
\[
 H_X \to  \mathbb{P}^1_{q^{\ell}} \setminus (-1)^{r+1},
\]
it produces a family of product structure on
$\widehat{H}_X^{\perp_\ell} \otimes {\widehat{\mathbb{C}[\N]}}$.
At two special points $0$ and $\infty$, the Frobenius structure
specializes to the big quantum cohomology modulo extremal rays of
$X$ and of $X'$ respectively.
The term ``analytic continuation'' used in this paper can thus
be understood in this way.

\subsection{Semisimplicity of big quantum ring for local models}
\label{s-s-local}

Toric varieties admits a nice big torus action and its equivariant
cohomology ring is always semisimple, hence as a deformation the
equivariant big quantum cohomology ring (the Frobenius manifold)
is also semisimple. Givental's quantization formalism works in
the equivariant setting, hence one way to prove the higher genus
invariance for local models is to extend results in \cite{LLW} to
the equivariant setting. This can in principle be done, but here
we take a direct approach which requires no more work.

\begin{lemma} \label{s-s}
For $X = \mathbb{P}_{\mathbb{P}^r}(\scr O(-1)^{r + 1} \oplus \scr
O)$, $QH^*(X)$ is semisimple.
\end{lemma}

\begin{proof}
By \cite{CoKa}, the proof of Proposition 11.2.17 and \cite{LLW},
Lemma 5.2, the small quantum cohomology ring is given by Batyrev's
ring (though $X$ is only semi-Fano). Namely for $q_1 = q^\ell$ and
$q_2 = q^\gamma$,
$$
QH^*_{small}(X) \cong \mathbb{C}[h, \xi][q_1, q_2] / (h^{r+1} -
q_1(\xi - h)^{r + 1}, (\xi - h)^{r + 1}\xi - q_2 ).
$$
Solving the relations, we get the eigenvalues of the quantum
multiplications $h*$ and $\xi*$:
\begin{equation}
h = \eta^j \omega^i q_1^{\frac{1}{r+1}} q_2^{\frac{1}{r + 2}}(1 +
\omega^i q_1^{\frac{1}{r + 1}})^{-\frac{1}{r + 2}},\quad \xi =
\eta^j q_2^{\frac{1}{r + 2}} (1 + \omega^i q_1^{\frac{1}{r +
1}})^{\frac{r + 1}{r + 2}}
\end{equation}
for $i = 0, 1, \cdots, r$ and $j = 0, 1, \cdots, r + 1$. where
$\omega$ and $\eta$ are the $(r + 1)$-th and the $(r + 2)$-th root
of unity respectively. As these eigenvalues of $h*$ (resp.~$\xi*$)
are all different, we see that $h*$ and $\xi*$ are semisimple
operators, hence $QH^*_{small}(X)$ is semisimple.

This proves that the formal Frobenius manifold $(QH^*, *)$ is
semisimple at the origin $s = 0$. Since semisimplicity is an
open condition, the formal Frobenius manifold $QH^*(X)$ is also
semisimple.
\end{proof}

\begin{remark}
The Batyrev ring for any toric variety, whether or not equal to
the small quantum ring, is always semisimple.
\end{remark}

\subsection{Invariance of mixed invariants of special type}

\begin{proposition} \label{p:local}
For the local models, the correspondence $\T$ for a simple flop
induces, after the analytic continuation, an isomorphism of the
ancestor potentials.
\end{proposition}

\begin{proof}
Since a flop induces $K$-equivalence, by (\ref{Euler}) the Euler
vector fields of $X$ and $X'$ are identified under $\T$. By
Theorem \ref{QH-inv} and Lemma \ref{s-s}, $X$ and $X'$ give rise
to isomorphic semisimple \emph{conformal} formal Frobenius
manifolds over $\R$ (or rather $\R_{loc}$):
$$
QH^*(X) \cong QH^*(X')
$$
under $\T$. The first statement then follows from Theorem
\ref{quantization}, the quantization formula, since all the
quantities involved are uniquely determined by the underlying
abstract Frobenius structure.

To be more explicit, to compare $\T \mathscr{A}_X$ with
$\mathscr{A}_{X'}$ is equivalent to compare $\T
(\widehat{\Psi}_X^{-1} \widehat{R}_X) e^{\widehat{\mathbf{u}}/z}$
with $\widehat{\Psi}_{X'}^{-1} \widehat{R}_{X'}
e^{\widehat{\mathbf{u'}}/z}$, and $\T \bar{c}$ with $\bar{c}'$.
Recall that
\[
 \epsilon_i :=  \partial_{u^i}, \quad
 \tilde{\epsilon}_i := \frac{\epsilon_i}
 {\sqrt{(\epsilon_i, \epsilon_i)}}.
\]

\begin{lemma}
$\T$ sends canonical coordinates on $X$ to canonical coordinates
on $X'$: $\T \epsilon_i = \epsilon'_{i}$, $\T \tilde{\epsilon}_i =
\tilde{\epsilon}'_{i}$. Moreover, $\bar{c}$, $\Psi$ and
$\mathbf{u}$ transform covariantly under $\T$.
\end{lemma}

\begin{proof}
As $\T$ preserves the big quantum product, $\T$ sends idempotents
$\{ \epsilon_i \}$ to idempotents $\{ \epsilon'_{i} \}$. Since the
canonical coordinates are uniquely defined for \emph{conformal}
Frobenius manifolds (up to $S_N$ permutation which is fixed by
$\T$), $\T$ takes canonical coordinates on $X$ to those on $X'$.
Furthermore, $\T$ preserves the Poincar\'e pairing \cite{LLW},
hence that $\T \tilde{\epsilon}_i = \tilde{\epsilon}'_{i}$.

The $\T$ covariance of $\bar{c}(\s) = \frac{1}{48} \ln \det
(\epsilon_i, \epsilon_j)$, the matrix $\mathbf{u}_{ij} =
(\delta_{ij} u^i)$ and the matrix ${\Psi}_{\mu i} =
(T_{\mu}, \tilde{\epsilon}_i)$ also follow immediately. For
example,
\[
\T {\Psi}_{\mu i} = (\T T_{\mu}, \T \tilde{\epsilon}_i).
\]
\end{proof}

The lemma implies that the Darboux coordinate systems on $X$ and
$X'$ defined by canonical coordinates are compatible under $\T$.
By the definition of the quantization process (\ref{e:wq}), which
assigns differential operators $\partial/\partial q^i_k$'s in an
universal manner under a Darboux coordinate system, it clearly
commutes with $\T$. It is thus enough to prove the invariance of
the semi-classical counterparts, or equivalently the
``covariance'' of the corresponding matrix functions, under $\T$.
Note that all the invariance and covariance are up to an analytic
continuation.

Therefore, one is left with the proof of the covariance of the $R$
matrix under $\T$, after analytic continuation. Namely $\T R(\s) =
R'(\T \s)$.

This follows from the uniqueness of $R$ for semisimple formal
conformal Frobenius manifolds. To be explicit, recall that in the
proof of \cite{LP2}, Theorem 1, the formal series $R(s, z) =
\sum_{n = 0}^\infty R_n(s) z^n$ of the $R$ matrix is recursively
constructed by $R_0 = {\rm Id}$ and the following relation in
canonical coordinates:
\begin{equation}
(R_n)_{ij}(du^i - du^j) = [(\Psi d \Psi^{-1} + d)R_{n - 1}]_{ij}.
\end{equation}
Applying $\T$ to it, we get $\T R_n = R'_n$ by induction on $n$.
\end{proof}

In order to generalize Proposition~\ref{p:local} to simple flops
of general smooth varieties, which will be carried out in the next
section by degeneration analysis, we have to allow descendent
insertions at the infinity marked points, i.e.~those marked points
where the cohomology insertions come from $j_* H^*(E) \subset
H^*(X)$.

\begin{theorem} \label{t:local}
For the local models, the correspondence $\T$ for a simple flop
induces, after the analytic continuation, an isomorphism of the
generating functions of  mixed invariants of special type in the
stable range.
\end{theorem}

\begin{proof}
Using Proposition~\ref{p:local} and \ref{p:1.1} and by induction
on the power $k$ of descendent, the theorem is reduced to the case
of $g = 0$ and with exactly one descendent insertion. It is of the
form $\langle \tau_k \alpha, T_{\nu} \rangle_0 (s)$ with $k \ge 0$
and by our assumption $\alpha \in j_*H^*(E)$. (Notice that for $s
= 0$ this is not in the stable range.) This series is a formal sum
of subseries
$$
\langle \tau_k \alpha, T_{\nu}, T_{\mu_1}, \cdots T_{\mu_l}
\rangle_{0, 2 + l}
$$
which are sums over $\beta \in NE(\tilde E)$. Each such series
supports a unique $d_2 \ge 0$ in $\beta = d_1 \ell + d_2 \gamma$.
If $d_2 = 0$ then the series and its counterpart in $X' = \tilde
E'$ (which supports the same $d_2$) are both trivial since
$\alpha$ is supported in $E$. If $d_2 > 0$, then the invariance
follows from \cite{LLW}, Theorem 5.6.
\end{proof}

We will generalize the theorem into the form of Theorem
\ref{t:main-2} by removing the local model condition after we
discuss the degeneration formula.

\begin{remark}
By section \ref{s-s-local} and the proof of Proposition
\ref{p:local}, the canonical coordinates $u^i$'s, idempotents
$\epsilon_i$'s, hence the transition matrix $\Psi$ and the $R$
matrix can all be solved in some integral extension $\tilde
\R_{loc}$ of $\R_{loc}$. It is interesting to know whether all
genus $g$ ancestor $n$-point generating functions take value in
$\tilde \R_{loc}$ and $\T \langle \tau_{\bar{l}\,} \alpha
\rangle^X_g = \langle \tau_{\bar{l}\,} \T \alpha \rangle^{X'}_g$
in $\tilde \R_{loc}$. This is plausible from Theorem
\ref{quantization} since the quantization process requires no
further extensions. In fact the calculation in genus one in
the Appendix suggests that $\langle \tau_{\bar{l}\,} \alpha \rangle^X_g$
might belong to $\R_{loc}$.
\end{remark}

\section{Degeneration analysis} \label{s:degeneration}

Let $f:X \dasharrow X'$ be a simple $\mathbb{P}^r$ flop with $\T$
being the graph correspondence. To prove Theorem \ref{t:main-2},
we need to show that, up to analytic continuation,
$$
\T \langle \tau_{k, \bar{l}\,} \alpha \rangle^X_g = \langle
\tau_{k, \bar{l}\,} \T \alpha \rangle^{X'}_g
$$
for all $\tau_{k, \bar l\,} \alpha = (\tau_{k_1, \bar l_1}
\alpha_1, \ldots, \tau_{k_n, \bar l_n} \alpha_n)$ being of
$f$-special type and $g \ge 0$ with $2g + n \ge 3$.

We follow the same strategy employed in Section~4 of \cite{LLW}.
The two minor changes are
\begin{enumerate}
\item to generalize primary invariants to ancestors (and
descendents); \item to generalize genus zero invariants to
arbitrary genus.
\end{enumerate}
Since it is quite straightforward to make the necessary
modifications, we will simply comment on the necessary changes and
ask the interested readers to consult Section~4 of \cite{LLW} for
further details.

The first step is to apply degeneration to the normal cone
$$
W = {\rm Bl}_{Z \times \{0\}} X \times \mathbb{A}^1 \to
\mathbb{A}^1.
$$
$W_0 = Y_1 \cup Y_2$, $Y_1 = Y = {\rm Bl}_Z X
\mathop{\to}\limits^\phi X$ and $Y_2 = \tilde E = \mathbb{P}_Z
(N_{Z/X} \oplus \scr O) \mathop{\to}\limits^p Z$. $E = Y \cap
\tilde E$ is the $\phi$ exceptional divisor as well as the
infinity divisor of $\tilde E$.

Define the generating series for genus $g$ (connected) relative
invariants
\begin{equation}
\langle A \mid \e, \mu \rangle^{(\tilde{E}, E)}_g
 := \sum_{\beta_2 \in NE(\tilde E)}
 \frac{1}{|{\rm Aut}\,\mu|} \langle A \mid \e, \mu
 \rangle_{g,\beta_2}^{(\tilde E, E)}\,\q^{\beta_2}
\end{equation}
and the similar one with possibly disconnected domain curves
\begin{equation}
\langle A \mid \e, \mu \rangle^{\bullet (\tilde E, E)} :=
\sum_{\Gamma;\, \mu_\Gamma = \mu} \frac{1}{|{\rm Aut}\,\Gamma|}
\langle A \mid \e, \mu \rangle_{\Gamma}^{\bullet (\tilde E,
E)}\,\q^{\beta^\Gamma}\,\hbar^{g^\Gamma - |\Gamma|}.
\end{equation}
Here for connected invariants of genus $g$ we assign the
$\hbar$-weight $\hbar^{g - 1}$, while for disconnected ones we
simply assign the product weights.

Since the degeneration formula is really about the degeneration of
the virtual cycles, the ancestors and descendents obey the same
formula. Therefore, Proposition~4.6 of \cite{LLW} can be
generalized into the following form:

\begin{proposition}[Reduction to relative local models]  \label{p:4.6}
To prove
\[
\T \langle \tau_{k, \bar{l}\,} \alpha \rangle^X_g = \langle
\tau_{k, \bar{l}\,} \T \alpha \rangle^{X'}_g
\]
for all $\alpha$ and $k, l$, it suffices to show
\[
\T \langle \tau_{k, \bar{l}\,} A \mid \e, \mu
\rangle^{(\tilde{E},E)}_h = \langle \tau_{k, \bar{l}\,} \T A \mid
\e, \mu \rangle^{(\tilde{E}',E)}_h
\]
for all $A \in H^*(\tilde E)^n$, $k, l\in \mathbb{Z}_+^n$, $\e \in
H^*(E)^\rho$, contact type $\mu$, and all $h \le g$.
\end{proposition}

\begin{proof}
For the $n$-point mixed generating function
$$
\langle \tau_{k, \bar l\,}\alpha \rangle^X = \sum_g \langle
\tau_{k, \bar l\,}\alpha \rangle^X_g \,\hbar^{g - 1} =
\sum_{g;\,\beta \in NE(X)} \langle \tau_{k, \bar l\,}\alpha
\rangle^X_{g,\beta}\, \q^\beta\,\hbar^{g - 1},
$$
the degeneration formula gives (let $m(\mu) = \prod \mu_i$,
$C_\eta = m(\mu)/|{\rm Aut}\,\eta|$):
\begin{align*}
&\langle \tau_{k, \bar l\,}\alpha \rangle^X \\
&= \sum_{\beta \in NE(X)}\sum_{\eta \in \Omega_\beta} \sum_I
C_\eta \langle \tau_{k_1, \bar{l}_1}\alpha_1 \mid \e_I, \mu
\rangle_{\Gamma_1}^{\bullet (Y_1,E)} \langle \tau_{k_2,
\bar{l}_2}\alpha_2 \mid \e^I, \mu \rangle_{\Gamma_2}^{\bullet
(Y_2,E)}\,\q^{\phi^*\beta}\,
\hbar^{g - 1}\\
&= \sum_{\mu} \sum_I \sum_{\eta \in \Omega_\mu} C_\eta
\left(\langle \tau_{k_1, \bar{l}_1}\alpha_1 \mid \e_I, \mu
\rangle_{\Gamma_1}^{\bullet
(Y_1,E)}\,\q^{\beta_1}\,\hbar^{g^{\Gamma_1} - |\Gamma_1|}\right)
\\
&\qquad \qquad \qquad \qquad \times \left(\langle \tau_{k_2,
\bar{l}_2}\alpha_2 \mid \e^I, \mu \rangle_{\Gamma_2}^{\bullet
(Y_2,E)}\,\q^{\beta_2}\,\hbar^{g^{\Gamma_2} - |\Gamma_2|}\right)
\hbar^\rho,
\end{align*}
where we have used $g - 1 = \sum_i(g^{\Gamma_i} - |\Gamma_i|) +
\rho$ with $\rho$ being the number of contact points. Notice that
$\beta = \phi_* \beta_1 + p_* \beta_2$ and we identify
$q^{\beta_1} = q^{\phi_* \beta_1}$, $q^{\beta_2} = q^{p_*
\beta_2}$ throughout our degeneration analysis.

We consider also absolute invariants $\langle \tau_{k,
\bar{l}\,}\alpha \rangle^{\bullet X}$ with product weights in
$\hbar$. Then by comparing the order of automorphisms,
\begin{equation}
\langle \tau_{k, \bar{l}\,}\alpha \rangle^{\bullet X} = \sum_{\mu}
m(\mu) \sum_I \langle \tau_{k_1, \bar{l}_1}\alpha_1 \mid \e_I, \mu
\rangle^{\bullet (Y_1,E)} \langle \tau_{k_2, \bar{l}_2}\alpha_2
\mid \e^I, \mu \rangle^{\bullet (Y_2,E)}\,\hbar^\rho.
\end{equation}

To compare $\T \langle \tau_{k, \bar{l}\,}\alpha \rangle^{\bullet
X}$ and $\langle \tau_{k, \bar{l}\,}\T \alpha \rangle^{\bullet
X'}$, by \cite{LLW}, Proposition 4.4, we may assume that $\alpha_1
= \alpha_1'$ and $\alpha_2' = \T\alpha_2$. This choice of
cohomology liftings identifies the relative invariants of $(Y_1,
E)$ and those of $(Y'_1, E')$ with the same topological types. It
remains to compare
$$
\langle \tau_{k_2, \bar{l}_2}\alpha_2 \mid \e^I, \mu
\rangle^{\bullet (\tilde E, E)} \quad \mbox{and} \quad \langle
\tau_{k_2, \bar{l}_2}\alpha_2 \mid \e^I, \mu \rangle^{\bullet
(\tilde E', E)}.
$$

We further split the sum into connected invariants. Let
$\Gamma^\pi$ be a connected part with the contact order $\mu^\pi$
induced from $\mu$. Denote $P: \mu = \sum_{\pi \in P} \mu^\pi$ a
partition of $\mu$ and $P(\mu)$ the set of all such partitions.
Then
$$
\langle A \mid \s, \mu \rangle^{\bullet (\tilde E, E)} = \sum_{P
\in P(\mu)} \prod_{\pi \in P} \sum_{\Gamma^\pi} \frac{1}{|{\rm
Aut}\,\mu^\pi|} \langle A^\pi \mid \s^\pi, \mu^\pi
\rangle_{\Gamma^\pi}^{(\tilde E, E)}\,\q^{\beta^{\Gamma^\pi}}\,
\hbar^{g^{\Gamma^\pi} - 1}.
$$

In the summation over $\Gamma^\pi$, the only index to be summed
over is $\beta^{\Gamma^\pi}$ on $\tilde{E}$ and the genus. This
reduces the problem to $\langle A^\pi \mid \s^\pi, \mu^\pi
\rangle^{(\tilde{E}, E)}_g$.

Instead of working with all genera, the proposition follows from
the same argument by reduction modulo $\hbar^{g + 1}$.
\end{proof}

\begin{proposition}[Relative to absolute] \label{p:4.8}
For a simple flop $\tilde{E} \dashrightarrow \tilde{E}'$, to prove
\[
\T \langle \tau_{\bar{l}\,} A \mid \e, \mu
\rangle^{(\tilde{E},E)}_g = \langle \tau_{\bar{l}\,} \T A \mid \e,
\mu \rangle^{(\tilde{E}',E)}_g
\]
for all $A, l, \e, \mu$, it suffices to show for mixed invariants
of special type
\[
\T \langle \tau_{\bar{l}\,} A, \tau_{k\,} \e \rangle^{\tilde{E}}_h
= \langle \tau_{\bar{l}\,} \T A, \tau_{k\,} \e
\rangle^{\tilde{E}'}_h
\]
for all $A, l, \e$ and $k \in \mathbb{Z}_+^\rho$, and all $h \le
g$.
\end{proposition}

\begin{proof}[Sketch of Proof]
We apply degeneration to the normal cone for $Z \hookrightarrow
\tilde E$ to get $W \to \mathbb{A}^1$. Then $W_0 = Y_1 \cup Y_2$
with $\pi: Y_1 \cong \mathbb{P}_E(\mathscr{O}_E(-1, -1)\oplus
\mathscr{O}) \to E$ a $\mathbb{P}^1$ bundle and $Y_2 \cong \tilde
E$.

By induction on $g$ and then on $(|\mu|, n, \rho)$ with $\rho$ in
the reverse ordering, the same procedure used in the proof of
\cite{LLW}, Proposition 4.8 leads to
\begin{align*}
&\langle \tau_{\bar l\,} A, \tau_{\mu_1 - 1}\e_{i_1}, \ldots,
\tau_{\mu_\rho - 1} \e_{i_\rho} \rangle^{\bullet\tilde E}_g
= \sum_{\mu'} m(\mu') \times \\
&\quad \sum_{I'} \langle \tau_{\mu_1 - 1} \e_{i_1}, \ldots,
\tau_{\mu_\rho - 1} \e_{i_\rho} \mid \e^{I'}, \mu'
\rangle^{\bullet (Y_1, E)}_0 \langle \tau_{\bar l\,} A \mid
\e_{I'}, \mu' \rangle^{(\tilde E, E)}_g + R,
\end{align*}
where $R$ denotes the remaining terms which either have lower
genus or have total contact order smaller than $d_2 = |\mu| =
|\mu'|$ or have number of insertions fewer than $n$ on the
$(\tilde E, E)$ side or the invariants on $(\tilde E, E)$ are
disconnected ones.

For the main terms, the integrals on $(Y_1, E)$ are all fiber
integrals and this allows to conclude that there is a single top
order term in the sum given by
$$
C(\mu)\langle \tau_{\bar l\,} A \mid \e_I, \mu\rangle^{(\tilde E,
E)}
$$
with $C(\mu) \ne 0$. Thus the proposition follows by induction.

\end{proof}

\begin{proof}[Proof of Main Theorems]
We only need to prove Theorem \ref{t:main-2}.

By Proposition \ref{p:4.6}, the theorem is reduced to the relative
local case. Moreover, for any insertion $\tau_{k,
\bar{l}\,}\alpha$ with nontrivial descendent ($k \ge 1$), we may
select the cohomology lifting of $\alpha$ to be $(\alpha, 0)$. To
avoid trivial invariants this insertion must go to the $(Y_1, E)$
side in the degeneration formula. Hence the theorem is reduced to
the case of relative local models $X = \tilde E =
\mathbb{P}_{\mathbb{P}^r} (\mathscr{O}(-1)^{r+1} \oplus
\mathscr{O})$ with at most ancestor insertions.

Now by Propositions \ref{p:4.8}, the proof is further reduced to
the case  for absolute invariants of the form
$$
\langle \tau_{\bar{l}\,} A, \tau_{k\,} \e \rangle^{\tilde{E}}_g
$$
which are mixed invariants of special type. But this is exactly
the content of Theorem \ref{t:local}. The proof is complete.
\end{proof}

\begin{remark}
The proof also shows that nontrivial descendent invariants of
$f$-special type without $2g + n \ge 3$, that is $(g, n) = (0, 1)$
or $(0, 2)$, are also invariant under $\T$.
\end{remark}

\section{Explicit formulae for primary invariants\\
attached to the extremal ray}

In this section we specialize our curve classes 
to the extremal ray and investigate the invariance in more explicit terms.
Note that all Gromov--Witten invariants discussed here are primary.

\subsection{Generalities concerning flopping curves}
\label{background}

Let $\dim X \ge 3$. In general, for $\ell$ being a curve class
with $(K_X.\ell) = 0$, the virtual dimension of
$\Mbar_{g, n}(X, d\ell)$ is given by
\begin{equation}
D_{g, n} = (\dim X - 3)(1 - g) + n
\end{equation}
which is independent of $d$. If moreover $\ell$ is a log-extremal
ray of flopping type (e.g.~in our case $\ell$ is the line class of
$Z$), $\langle \alpha \rangle_{g,n, d\ell}$ depends only on the
local geometry of $(Z, N_{Z/X})$ for $ d\ge 1$. But for $d = 0$ it
depends on the global geometry of $X$.

If $D_{g,n}$ is negative, all Gromov--Witten invariants must vanish.
From fundamental class axiom, all primary invariants
$\langle 1, \cdots \rangle_{g,n,d\ell}$ must vanish if
$\Mbar_{g,n-1}(X,d\ell)$ exists. We are therefore left with 3 cases:
$g=0$, $g=1$, $\dim X =3$ (and $g\ge 2$).

\vspace{5pt}
$g = 0$ then $D_{g, 3} = \dim X$ and the 3-point functions with
$d \ge 1$ are expected to correct the classical cubic product
corresponding to $d = 0$, which is indeed the case for simple
flops. There is no $n$-point invariant with $n \ge 4$ and $d = 0$.
In fact for simple flops the $n$-point functions with $n \ge 4$,
$d \ge 1$ are invariant under $\T$.

\vspace{5pt}
If $g = 1$ then $D_{1, n} = n$. By the fundamental class axiom
each cohomology insertion must be a divisor.
Hence if $d \ge 1$ by the divisor axiom the $n$-point
invariants are determined by the ``partition function''
$$
\int_{[\Mbar_{1,0}(X, d \ell)]^{\vir}} 1.
$$
For $d = 0$ and $n \ge 2$, the divisor axiom shows that $\langle \alpha
\rangle_{1, n, 0} = 0$.
$n = 1$ case requires different consideration.

Indeed it is well known that $\overline M_{g, n}(X, 0) \cong X
\times \overline M_{g, n}$ and the virtual fundamental class is
given by
\begin{equation} \label{virt-0}
[\overline M_{g, n}(X, 0)]^{vir} = e(\mathcal{E})\cap [X \times
\overline M_{g, n}]
\end{equation}
where the obstruction bundle is given by $\mathcal{E} = \pi_1^*
T_X \otimes \pi_2^* \mathcal{H}_g^\vee$ with $\mathcal{H}_g$ being
the Hodge bundle. Let $\lambda_i = c_i(\mathcal{H}_g)$.

For $(g, n) = (1, 1)$ we clearly have
$$
c(\mathcal{E}) = c_{\rm top}(X) - c_{{\rm top} - 1}(X).\lambda_1.
$$
Thus for one point invariants we get a (semi-)classical term
\begin{equation} \label{g=1-defect}
\langle \alpha \rangle_{1,0}^X = -(c_{{\rm top} - 1}(X).\alpha)_X
\cdot \int_{\overline M_{1,1}} \lambda_1 = -\frac{1}{24}(c_{{\rm
top} - 1}(X).\alpha)_X,
\end{equation}
where the basic Hodge integral $\int_{\overline{M}_{1,1}}
\lambda_1 = 1/24$ is used.

For simple flops, we will verify that $\T\langle \alpha
\rangle_1^X = \langle \T\alpha \rangle_1^{X'}$ in the next
subsection by showing that the genus one invariants with $d \ge 1$
correct the semi-classical defect $\langle \alpha \rangle_{1,0}^X
- \langle \T\alpha \rangle_{1,0}^{X'}$.

\vspace{5pt}
For $\dim X = 3$ (and $g \ge 2$), $D_{g, n} = n$.
As in the $g = 1$ case we are reduced to consider the case $n = 0$.
For simple $\mathbb{P}^1$ flop, the extremal invariants with $d \ge 1$ are
determined by Faber and Pandharipande \cite{FP} to be
$$
\langle - \rangle_{g, d} := \int_{[\Mbar_{g,0}(X, d \ell)]^{\vir}} 1 = C_g\, d^{2g - 3}
$$
where $C_g = |\chi(M_g)|/(2g - 3)!$.
We claim that the generating function
\begin{equation}
\langle - \rangle_g := \sum_{d = 0}^\infty \langle - \rangle_{g,
d}\, q^{d} = \langle - \rangle_{g, 0} + C_g\, \delta^{2g - 3} \f,
\end{equation}
is invariant under $\T$ (since $2g - 3 \ge 1$), where the operator
$\delta$ is defined in \eqref{e:delta}.
The second term is invariant following the analysis in Section~\ref{s:3}.
For degree zero term, it is not difficult to see
$\langle - \rangle_{g,0}^X = \langle - \rangle_{g,0}^{X'}$:
The degeneration analysis in Section~\ref{s:degeneration} reduces
the proof to a corresponding statement for local models.
The local models of $X$ and $X'$ are both isomorphic to
$\mathbb{P}_{\mathbb{P}^1}(\scr O(-1)^2 \oplus \scr O)$, and hence
have the same degree zero invariants.
In fact it is not hard to see from (\ref{virt-0}) that
\begin{equation}
e(\mathcal{E}) = \frac{(-1)^g}{2} \big( c_3(X) - c_2(X).c_1(X)
\big)\lambda^3_{g - 1}.
\end{equation}
The invariance of these Chern numbers can be directly verified for
$\mathbb{P}^1$ flops.

\subsection{The genus one case}

Let $G$ be the genus one potential of twisted GW invariants on
$\scr{O}_{\bb P^r}(-1)^{r+1}$ without marked points. Namely, for
$\ell$ being the line class in $\mathbb{P}^r$,
$$
G(q) := \langle - \rangle_1 = \sum_{d \ge 0} \langle - \rangle_{1,
d\ell}\, q^d.
$$

For $r = 1$, $d \ge 1$, the formula
$$
\langle - \rangle_{1, d} = \frac{1}{12d}
$$
was first obtained by physical consideration in \cite{BCOV} and
later mathematically justified in \cite{GP}. Here, by using the
theory of semisimple Frobenius manifolds, we generalize it to all
$r \in \mathbb{N}$ and get
\begin{proposition}
For $d \in \mathbb{N}$,
\begin{equation} \label{g=1-general r}
\langle - \rangle_{1, d} = (-1)^{d(r + 1)}\frac{r + 1}{24d}.
\end{equation}
\end{proposition}

In fact the $g = 1$ case can be achieved without using the
machinery of quantization. Givental has shown in \cite{Gi},
Theorem 4.1 that the total differential $dG$ can be expressed, up
to a constant, in terms of the canonical coordinates $u_i$'s as
\begin{equation} \label{dG}
dG = \sum_i \left[d \log \Delta_i/48 - c^i_{-1} du_i/24 + R^1_{ii}
du_i/2\right],
\end{equation}
where the RHS is determined by the {\em equivariant genus zero
theory}.

In the appendix we will determine all the terms $u_i$, $\Delta_i
:= 1/(\epsilon_i.\epsilon_i)$ and $R^1_{ii}$ step by step. The
final result (Theorem~\ref{t:g1}) is equivalent to
(\ref{g=1-general r}):
\begin{equation} \label{final}
\delta_h G = (-1)^{r + 1} \frac{r+1}{24} \f
\end{equation}
in the {\em non-equivariant limit}, where $\delta_h := q d/dq$.

We use it in the following setting: Let $f: X \dasharrow X'$ be a
simple $\mathbb{P}^r$ flop. Under the canonical correspondence $\T
= [\bar\Gamma_f]$ we have  $\T\ell = -\ell'$ and we identify $q' =
q^{-1}$. Then $\delta_{h'} = -\delta_h$ and by (\ref{f-eq}),
$\delta_h \f(q) = \delta_{h'} \f(q')$. Hence $\delta_h^2 G(q) =
\delta_{h'}^2 G'(q')$ and by the divisor axiom
$$
\langle h, \ldots, h \rangle_{1, n}^X = \delta_h^n G(q) = (-1)^{n
- 2} \delta_{h'}^n G'(q') = (-1)^n \langle h', \ldots, h'
\rangle_{1, n}^{X'}.
$$

Since $\T h^k = (-1)^{k} h'^k$, this implies the invariance of
$n$-point functions for all $n \ge 2$. (The invariance for $g \ge
2$, $n \ge 1$ is proved in the same way.)

For $n = 1$, to prove the invariance we may assume that $X$ and
$X'$ are local models. We compute via (\ref{g=1-defect}),
(\ref{final}) and (\ref{f-eq}) that
$$
\langle h \rangle_1^X - \langle \T h \rangle_1^{X'} =
-\frac{1}{24}\left((c_{2r}(X).h) - (c_{2r}(X').(\xi' - h'))
\right) - \frac{r + 1}{24}.
$$
Since $X \cong X'$, the invariance will follow from
\begin{proposition}
$$(c_{2r}(X).(2h - \xi)) = -(r + 1).$$
\end{proposition}

\begin{proof}
Since $X = \mathbb{P}_Z (N_{Z/X} \oplus \scr O)
\mathop{\to}\limits^p Z$, by the Leray-Hirsh theorem,
$$
H^*(X) = \mathbb{Z}[h, \xi]/(h^{r + 1}, \xi(\xi - h)^{r + 1}).
$$

From $0 \to \scr O \to \scr O(1)\otimes p^*(N_{Z/X} \oplus \scr O)
\to T_{X/Z} \to 0$ we get $c(X) = p^*c(Z).c(T_{X/Z}) = (1 + h)^{r
+ 1}(1 + \xi)(1 + \xi - h)^{r + 1}$. So we need to calculate the
degree $2r + 1$ terms in
$$
c(X).(2h - \xi) = (1 + h)^{r + 1}(1 + \xi)(1 + \xi - h)^{r + 1}(2h
- \xi)
$$
under the additional relation $h^r \xi^{r + 1} = 1$.

Since $\xi(\xi - h)^{r + 1} = 0$, $(1 + \xi)(1 + (\xi - h))^{r +
1}$ start with
$$
(\xi - h)^{r + 1} + C^{r + 1}_1 \xi(\xi - h)^r + C^{r + 1}(\xi -
h)^r + C^{r + 1}_2 \xi(\xi - h)^{r - 1} + \mbox{lower terms}.
$$
Thus in $(1 + h)^{r + 1} = h^r + C^{r + 1}_1 h^{r - 1} + \cdots$
only $h^r$ and $h^{r - 1}$ contribute to terms with total degree
$2r + 1$. There are four such terms:
$$
C^{r + 1}_1 C^{r + 1}_1 h^r(\xi - h)^r (2h - \xi) = -(C^{r +
1}_1)^2,
$$
$$
C^{r + 1}_1 C^{r + 1}_2 h^r \xi(\xi - h)^{r - 1}(2h - \xi) = -C^{r
+ 1}_1 C^{r + 1}_2,
$$
$$
C^{r + 1}_2 h^{r - 1}(\xi - h)^{r + 1}(2h - \xi) = 2C^{r + 1}_2
$$
and (using the Chern relation $0 = \xi(\xi - h)^{r + 1} = \xi^{r +
1} -C^{r + 1}_1 \xi^r h + \cdots$)
\begin{align*}
&C^{r + 1}_2 C^{r + 1}_1 h^{r - 1}\xi(\xi - h)^r (2h - \xi)= C^{r
+ 1}_2 C^{r + 1}_1 (2h^r \xi(\xi - h)^r -h^{r -
1}\xi^2(\xi - h)^r)\\
&= C^{r + 1}_2 C^{r + 1}_1 (2 - C^{r + 1}_1 + C^r_1) = C^{r + 1}_2
C^{r + 1}_1.
\end{align*}
The sum is $-(r + 1)^2 + (r + 1)r = -(r + 1)$ as expected.
\end{proof}

\section{Appendix: The calculations for $g = 1$}

We calculate the genus one twisted Gromov--Witten invariants on
the bundle $\scr{O}_{\bb P^r}(-1)^{r+1}$ by using Givental's work
\cite{Gi}.

\subsection{The Frobenius structure and the canonical coordinates}

This calculation follows the general scheme outlined in \cite{Gi}.

Let $E$ be the total space of $\scr{O}_{\bb P^r}(-1)^{r+1}$ over
$\mathbb{P}^r$ and consider the torus action of $\mathbb{C}^*
\times \mathbb{C}^* $ on $E$ such that the first $\mathbb{C}^*$
acts trivially on $\mathbb{P}^r$ and by scalar multiplication on
the fiber of $E$ and the second one acts by
$$
\alpha \cdot [x_0:x_1:\cdots :x_r] = [\alpha^{l_0}
x_0:\alpha^{l_1} x_1:\cdots :\alpha^{l_r} x_r].
$$
Let $\lambda$ and $\lambda'$ denote the characters of these two
actions respectively.  Let $i: E_{fixed} \hookrightarrow E$ be the
injection, where $E_{fixed}$ denotes the fixed loci. Denote by $p$
the equivariant hyperplane class of $\bb P^r$ for the first
action.

\begin{proposition} \label{char}
The characteristic polynomial of $p*$ in equivariant small quantum
cohomology is given by
\begin{equation} \label{char-poly}
q(\lambda -p)^{r + 1} = p^{r + 1}.
\end{equation}
\end{proposition}

\begin{proof}
Formally this follows from the formula for small quantum ring of
local models in Lemma \ref{s-s}, with $h$, $\xi$ being replaced by
$p$, $\lambda$.

Alternatively, by \cite{Gi} Corollary 4.4 (in the limit $\lambda'
\rightarrow 0$ which exists),
$$
\mbox{\small Quantum}\,(\lambda -p)^{r+1} = \mbox{\small
Classical}\,\frac{1}{1+(-1)^rq}(\lambda -p)^{r+1}.
$$

Using the fact that $p^{r+1}=0$ in the classical product and the
quantum $p^k$ coincides with the classical one for $k \le r$, we
get
$$
(1 + (-1)^r q)(\lambda - p)^{r + 1} = (\lambda - p)^{r + 1} -
(-1)^{r + 1}p^{r + 1}
$$
in small quantum product. The proposition follows.
\end{proof}

Solving this formally in $q$ or locally analytically in the
K\"ahler moduli coordinate $t$ with $q = e^t$, we get (with $\xi
:= e^{2\pi i/(r + 1)}$):
\begin{equation} \label{roots}
p_i = \frac{\lambda}{1+(-1)^r\xi^i q^{-\frac{1}{r+1}}},\qquad i =
0, 1, \ldots, r.
\end{equation}

Using the {\em standard basis} $1, p, \ldots, p^r$ of $H^*_{{\bb
C}^\times}({\bb P}^r)$, it is easy to see that $p_i$'s are the
eigenvalues of the quantum multiplication operator $p*$. It
follows that for $k \le r$, $p_i^k$'s are the eigenvalues of $p^k*
= (p*)^{\circ k}$. The common eigenvectors $\epsilon_i$'s which
simultaneously diagonalize the quantum product $p^k * \epsilon_i =
p^k_i \epsilon_i$ are known as the {\em canonical basis}. Let
$t_i$'s be the {\em standard (flat) coordinates} dual to $\{1, p,
\ldots, p^r\}$ and $u_i$'s be the {\em canonical coordinates} dual
to $\{\epsilon_0, \ldots, \epsilon_r\}$. The canonical basis
$\{\epsilon_i\}$ is {\em orthogonal} with respect to the
Poincar\'e pairing since quantum multiplications are self adjoint.

In practice we may simply define $u_i$ and then $\epsilon_i$ by
the relation
\begin{equation} \label{canonical}
du_i = \sum_{k = 0}^r p_i^k \, dt_k.
\end{equation}
Calculations in canonical coordinates are thus essentially {\em
formal linear algebra} which in our case are reduced to the {\em
roots-coefficients} relation
\begin{equation} \label{char-poly2}
(-1)^r p^{r + 1} - \f C^{r + 1}_1 \lambda p^r + \f C^{r + 1}_2
\lambda^2 p^{r - 1} - \cdots + (-1)^{r + 1} \f \lambda^{r + 1} =
0.
\end{equation}

\begin{remark}
It is important to point out that all the coefficients involve
$\f$ only! Thus {\em in principle} everything {\em canonically
determined} by the Frobenius structure should be invariant under
flops after one more differentiation by $\delta_h = qd/dq \equiv
d/dt$. This note provides a demonstration in this direction.
\end{remark}

We start with determining $\sum_{i}c^i_{-1}du_i/24$. Recall that
$c^i_{-1}$ is the localization of $c_{\dim - 1}/c_{\dim}$ at the
$i$-th fixed point $[0:\cdots:1:\cdots:0]$. Since the Chern roots
at there are given by
\begin{align*}
\underbrace{l_i\lambda'+\lambda,\cdots,
l_i\lambda'+\lambda}_{r+1}, \quad
\underbrace{(l_j-l_i)\lambda'}_{r} \quad(j\neq i),
\end{align*}
we see that
$$
c^i_{-1} = \frac{r+1}{l_i\lambda'+\lambda}+\sum_{j\neq
i}\frac{1}{(l_j-l_i)\lambda'}.
$$
In the limit $\lambda' \rightarrow 0$, the trouble terms with
$1/\lambda'$ must be canceled out with the corresponding terms in
$R^1_{ii}$'s via (\ref{dG}). So $c^i_{-1} = (r+1)/\lambda$ and
\begin{equation}
\sum_{i}c^i_{-1} du_i/24 = \frac{r + 1}{24\lambda} \sum_{i} du_i =
\frac{r + 1}{24\lambda} \sum_{k = 0}^r \Big(\sum_{i} p_i^k\Big)
dt_k.
\end{equation}

Since we are only interested in the non-equivariant limit $\lambda
\to 0$, by (\ref{char-poly2}) the only terms remaining are with $k
= 1$, hence we obtain
\begin{proposition} \label{2nd_term}
$$
\sum_{i}c^i_{-1} du_i/24 = (-1)^r \frac{(r + 1)^2}{24} \f\, dt_1.
$$
\end{proposition}

In Givental's formulation, we will need to identify $t = t_1$ in
the sequel. But we will keep $q = e^t$ independent of $t_1$
whenever possible.

\subsection{The Poincar\'e pairing and the formula for
$\Delta_i := \langle \epsilon_i.\epsilon_i \rangle^{-1}$} {\ }

\begin{lemma} \label{poincare}
The equivariant Poincar\'e pairing is given by
$$
\langle p^k.p^l\rangle = C^{2r - d}_{r - d} \frac{1}{\lambda^{2r +
1 - d}}
$$
where $d = k + l$. It vanishes if $d > r$.
\end{lemma}

\begin{proof}
The localization formula for the first $\mathbb{C}^\times$ action
reads as:
$$
\int_{E} \omega=\int_{\mathbb{P}^r}
\frac{i^*\omega}{(\lambda-p)^{r+1}}=\int_{\mathbb{P}^r}
\frac{i^*\omega}{\lambda^{r+1}}\Big(1 - \frac{p}{\lambda}
\Big)^{-(r+1)}.
$$
Since $p^k$ vanishes if $k > r$, we get by Taylor expansion that
$$
\Big(1 - \frac{p}{\lambda}\Big)^{-(r + 1)} = 1 + C^{r + 1}_{1}
\frac{p}{\lambda} + C^{r + 2}_{2} \Big(\frac{p}{\lambda}\Big)^2 +
\cdots + C^{r + r}_{r} \Big(\frac{p}{\lambda}\Big)^r.
$$
So the Poincar\'e pairing is given by
$$
\langle p^k.p^l\rangle = \int_E p^d = C^{r + r - d}_{r - d}
\frac{1}{\lambda^{r + 1 + (r - d)}} = C^{2r - d}_{r - d}
\frac{1}{\lambda^{2r + 1 - d}}.
$$
\end{proof}

Define $a_i := 1 + (-1)^r \xi^i q^{-\frac{1}{r+1}}$ so $p_i =
\lambda/a_i$. For convenient we also denote by $c_i = (-1)^r \xi^i
q^{-\frac{1}{r+1}}$, so $a_i = 1 + c_i$ and $c_j = c_0 \xi^j = c_i
\xi^{j - i}$.

\begin{proposition} \label{epsilon}
The canonical basis $\epsilon_i$'s are given by
$$
\epsilon_i = \frac{q  c_i}{r + 1} a_i^{r} \prod_{l\neq i} \Big(1 -
a_l \frac{p}{\lambda}\Big).
$$
They are vector fields along the K\"ahler moduli variable $q$.
\end{proposition}

\begin{proof}
Call the RHS $\varepsilon_i$ and it suffices to show that
$du_j(\varepsilon_i) =\delta_{ji}$. By (\ref{canonical}), the
effect of $du_j$ is simply the replacement of $p$ by $p_j =
\lambda/a_j$. Hence
\begin{align*}
du_j(\varepsilon_i) &= \frac{q c_i}{r+1}
\Big(\frac{a_i}{a_j}\Big)^r \prod_{l\neq i} (a_j - a_l)\\
&= \frac{q c_i}{r+1}  \Big(\frac{a_i}{a_j}\Big)^r
\prod_{l\neq i} (-1)^r (\xi^j - \xi^l) q^{-\frac{1}{r+1}} \\
&= \frac{\xi^{jr+i}}{r+1}  \Big(\frac{a_i}{a_j}\Big)^r
\prod_{l\neq i} (1-\xi^{l-j}) = \left \{ \begin{array}{ll}
1 & \textrm{if $j=i$}\\
0 & \textrm{if $j \neq i$}
\end{array} \right..
\end{align*}
\end{proof}

Denote by $S^l_k(x)$ be the $k$-th elementary symmetric polynomial
in $x_j$'s with $0\le j \le r$ and $l \ne l$. We often need some
basic formulae on roots of unity whose verifications are
elementary and omitted.

\begin{lemma} \label{trivial}
The following formulae for $\xi$ hold:
\begin{itemize}
\item[(1)] $S^i_k(\xi) := S^i_k(\xi^0, \ldots, \xi^r) = (-1)^k
\xi^{ki}$ for $k = 0,1,\ldots, r$.

\item[(2)]
$\sum_{i=0}^{r}\xi^{ki}=0$ for $k = 1, 2, \ldots, r$.
\end{itemize}
\end{lemma}

Clearly the lemma applies to $c_j$'s as well. Then
\begin{equation} \label{epsilon2}
\begin{split}
\epsilon_i &= \frac{q c_i a_i^{r}}{r + 1} \prod_{j\neq i}
\left[\Big(1 - \frac{p}{\lambda}\Big) - c_j
\frac{p}{\lambda}\right]  \\
&= \frac{q c_i a_i^{r}}{r + 1} \sum_{k = 0}^r (-1)^k S^i_k(c)
\Big(\frac{p}{\lambda}\Big)^k \Big(1 - \frac{p}{\lambda}\Big)^{r -
k} \\
&= \frac{q c_i a_i^{r}}{r + 1} \sum_{k = 0}^r c_i^k
\Big(\frac{p}{\lambda}\Big)^k \Big(1 - \frac{p}{\lambda}\Big)^{r -
k}.
\end{split}
\end{equation}

It is convenient to denote $\langle p \rangle^{k + l} = \langle
p^k.p^l \rangle = \langle p^{k + l} \rangle$.

\begin{lemma} \label{zero}
For $k = 0, 1, \ldots, r - 1$,
$$
\Big< \Big(\frac{p}{\lambda}\Big)^k \Big(1 -
\frac{p}{\lambda}\Big)^{2r - k} \Big> = 0.
$$
\end{lemma}

\begin{proof}
By Lemma \ref{poincare}, it equals
\begin{align*}
&\quad \Big<\frac{p}{\lambda}\Big>^k-C^{2r - k}_{1}
\Big<\frac{p}{\lambda}\Big>^{k + 1} + \cdots + (-1)^{r - k}C^{2r -
k}_{r - k}
\Big<\frac{p}{\lambda}\Big>^{k + (r - k)}\\
&= \frac{1}{\lambda^{2r + 1}} \left(C^{2r - k}_{r - k} - C^{2r -
k}_{1} C^{2r - (k + 1)}_{r - (k + 1)} + \cdots + (-1)^{r - k}
C^{2r - k}_{r - k} C^{2r - (k + (r - k))}_{r - (k
+ (r - k))} \right)\\
&= \frac{C^{2r - k}_r}{\lambda^{2r + 1}} \left(1 - C^{r - k}_{1} +
\cdots + (-1)^{r - k} C^{r - k}_{r - k} \right) \\
&= \frac{C^{2r - k}_r} {\lambda^{2r + 1}}(1 - 1)^{r - k} = 0.
\end{align*}
\end{proof}

Using (\ref{epsilon2}), Lemma \ref{trivial} and Lemma \ref{zero},
we compute
\begin{align*}
\Delta_i^{-1} &= \langle \epsilon_i.\epsilon_i \rangle \\
&= \frac{q^2 c_i^2 a_i^{2r}} {(r+1)^2}(-1)^r (r+1) c_i^r \Big<
\Big(\frac{p}{\lambda}\Big)^r
\Big(1 - \frac{p}{\lambda} \Big)^r \Big>\\
&= \frac{q^2 c_i^2 a_i^{2r}} {r+1} c_i^r \Big<\frac{p}{\lambda}
\Big>^r \qquad (c_i^{r + 1} = q^{-1})\\
&= \frac{q c_i a_i^{2r}}{(r+1)\lambda^{2r+1}}.
\end{align*}

\begin{lemma} \label{Delta_i}
As a function in $q$, the norm square inverse of $\epsilon_i$
equals
$$
\Delta_i = (r+1)\lambda q^{-1} c_i^{-1} p_i^{2r}.
$$
\end{lemma}

\begin{proposition} \label{1st_term}
$$
d\log (\Delta_0 \Delta_1 \cdots \Delta_r) = r
\frac{1-(-1)^rq}{1+(-1)^rq}\, d\log q = r(1 - 2(-1)^r \f)\, d\log
q.
$$
\end{proposition}

\begin{proof}
Simply take log differentiation of
$$
\Delta_0 \Delta_1 \cdots \Delta_r = (r+1)^{r+1}
\lambda^{(2r+1)(r+1)} \xi^{-\frac{r(r+1)}{2}}  q^{-r} \f^{2r}.
$$
\end{proof}

\subsection{The transition matrix $\Psi$ and the connection one
form $\Psi d\Psi^{-1}$} {\ }

The matrix $\Psi$ is defined by
$$
p^\mu = \sum_i \Psi_{\mu}^i \tilde \epsilon_i = \sum_i \langle
p^\mu.\tilde \epsilon_i \rangle \tilde \epsilon_i
$$
relative to the orthonormal frame $\{\tilde \epsilon_i :=
\sqrt{\Delta_i}\epsilon_i\}$. From (\ref{canonical}), we get the
dual relation
\begin{equation} \label{canonical2}
p^k = \sum_i p^k_i \epsilon_i.
\end{equation}
Hence by Lemma \ref{Delta_i},
\begin{equation}
\Psi_{\mu}^i = p^\mu_i/ \sqrt{\Delta_i} = \sqrt{\frac{qc_i}{r +
1}} \lambda^{-\frac{1}{2}} p_i^{\mu - r}\\
= \sqrt{\frac{qc_i}{r + 1}} \lambda^{\mu - r - \frac{1}{2}} a_i^{r
- \mu}.
\end{equation}

The inverse $\Psi^{-1}$ has already been determined in Proposition
\ref{epsilon} up to a normalization factor. Indeed,
$(\Psi^{-1})_j^\mu$ is the coefficient of $p^\mu$ in the
expression of $\sqrt{\Delta_j} \epsilon_j$, hence we get
\begin{equation}
(\Psi^{-1})_j^{\mu} = (-1)^{\mu} \sqrt{\frac{qc_j}{r + 1}}
\lambda^{r - \mu + \frac{1}{2}} S^j_\mu(a).
\end{equation}

To proceed, notice that by Lemma \ref{trivial}
\begin{align*}
S^j_\mu(a) &= \sum_{k_1 < \cdots < k_\mu,\,k_s \ne j} (1 +
c_{k_1}) \cdots (1 + c_{k_{\mu}})\\
&= \sum \Big(1 + (c_{k_1} + \cdots c_{k_\mu}) + (c_{k_1}c_{k_2} +
\cdots) + \cdots + (c_{k_1}\cdots c_{k_\mu}) \Big)\\
&= C^r_\mu - C^{r - 1}_{\mu - 1}c_j + C^{r - 2}_{\mu - 2}c_j^2 -
\cdots + (-1)^\mu c_j^\mu.
\end{align*}

We regard $q = e^t$ and take differentiation in $t = \log q$. Then
$$
(d\Psi^{-1})_j^{\mu} = \frac{r}{2(r + 1)} (\Psi^{-1})_j^{\mu}\, dt
+ (-1)^{\mu} \sqrt{\frac{qc_j}{r + 1}} \lambda^{r - \mu +
\frac{1}{2}}\, d S_\mu^j(a),
$$
and we compute
\begin{align*}
(\Psi d\Psi^{-1})_j^i &= \sum_{\mu=0}^{r}\Psi_{\mu}^{i}
(d\Psi^{-1})_{j}^{\mu} = \frac{r\delta_{ij}}{2(r + 1)} \,dt
- \frac{q \sqrt{c_i c_j}}{(r + 1)^2}\,dt \\
&\times  \sum_{\mu = 1}^r (-1)^{\mu} a_i^{r - \mu}\textbf{}
\Big(-C^{r-1}_{\mu-1} c_j + 2C^{r - 2}_{\mu - 2} c_j^2 - \cdots +
(-1)^{\mu} \mu c_j^{\mu}\Big).
\end{align*}
The last sum equals
\begin{align*}
&\quad c_j \Big(a_i^{r-1} - C^{r-1}_{1} a_i^{r-2} +
+ \cdots + (-1)^{r-2} C^{r-1}_{r-2} a_i + (-1)^{r-1} \Big)\\
&\qquad +2c_j^2 \Big( a_i^{r-2} - C^{r-2}_{1} a_i^{r-3} +
\cdots + (-1)^{r - 3} C^{r-2}_{r-3} a_i + (-1)^{r - 2} \Big) + \cdots\\
&= c_j(a_i - 1)^{r - 1} + 2c_j^2 (a_i - 1)^{r - 2} + \cdots +rc_j^r\\
&= c_j c_i^{r-1} + 2c_j^2 c_i^{r - 2} + \cdots + rc_j^r.
\end{align*}

If $i = j$, using $q c_j^{r + 1} = 1$ we get
$$
(\Psi d\Psi^{-1})_i^i = \frac{r}{2(r + 1)} \,dt - \frac{r(r +
1)}{2} \frac{qc_j^{r + 1}}{(r + 1)^2} \,dt = 0.
$$
This agrees with the well-known fact that $\Psi d\Psi^{-1}$ is
skew-symmetric.

If $i \ne j$, using $c_j = c_i \xi^{j - i}$ we get the connection
matrix to be
\begin{equation} \label{conn}
(\Psi d\Psi^{-1})^i_j = \frac{\xi^{\frac{1}{2}(j - i)}}{(r + 1)^2}
\sum_{\mu = 1}^r \mu \xi^{\mu(j - i)}\,d\log q.
\end{equation}

\subsection{The first asymptotic matrix $R^1$ and the
final computation} {\ }

Now we identify the K\"ahler moduli coordinate $t = t_1$. Recall
the defining relation of $R^1$:
\begin{equation} \label{R1}
(\Psi d\Psi^{-1})_j^i = R^1_{ij}(du_i - du_j).
\end{equation}
The off-diagonal part of $R^1$ is uniquely solvable from
(\ref{R1}). Since the LHS involves only $d\log q = dt$ and $du_i -
du_j = \sum_{k = 0}^r (p_i^k - p_j^k)\,dt_k$, we get
\begin{equation} \label{R1-off}
R^1_{ij} = \frac{(-1)^r \xi^{\frac{1}{2}(j - i)}}{(r +
1)^2\lambda(\xi^{j}-\xi^{i})} q^{\frac{1}{r + 1}} a_ia_j \sum_{\mu
= 1}^r \mu \xi^{\mu(j - i)} .
\end{equation}

Denote by
$$
g_k(\xi) := \frac{1}{\xi^k-1} \sum_{\mu = 1}^r \mu \xi^{\mu k}
\sum_{\mu = 1}^r \mu \xi^{-\mu k}.
$$

The diagonal part is determined by the {\em flatness relation} up
to a constant:
\begin{equation} \label{dR1}
dR^1_{ii} +\sum_{j} R^1_{ij} R^1_{ji}(du_i - du_j) = 0.
\end{equation}
Hence by (\ref{conn}), (\ref{R1-off}) and writing out $a_ia_j$ we
get
\begin{align*}
dR^1_{ii} &=\sum_{j\neq i}\frac{(-1)^r
q^{\frac{1}{r+1}}a_ia_j}{(r+1)^4\lambda(\xi^{j}-\xi^{i})} \, dt
\sum_{\mu = 1}^r \mu \xi^{\mu(j - i)}
\sum_{\mu = 1}^r \mu \xi^{\mu(i - j)} \\
&=\frac{(-1)^r \xi^{-i} q^{\frac{1}{r+1}}}{(r+1)^4\lambda}\,dt
\sum_{k=1}^r g_k(\xi)
+\frac{1}{(r+1)^4\lambda}\,dt \sum_{k=1}^r (\xi^{k}+1) g_k(\xi)\\
&\qquad + \frac{(-1)^r \xi^{i} q^{-\frac{1}{r+1}}}{(r+1)^4\lambda}
\,dt\sum_{k=1}^r \xi^{k} g_k(\xi).
\end{align*}

\begin{lemma}
\begin{align*}
&\sum_{k=1}^r (\xi^{k}+1) g_k(\xi) = 0.
\end{align*}
\end{lemma}

\begin{proof}
Let $f(\xi)$ be the expression, then $f(\xi) = f(\xi^{-1}) =
-f(\xi)$.
\end{proof}

Let
$$
\Xi_r:=\sum_{k=1}^r g_k(\xi).
$$
By the lemma,
\begin{align*}
dR^1_{ii} = \frac{(-1)^r \Xi_r}{(r+1)^4\lambda}(\xi^{-i}
q^{\frac{1}{r + 1}} - \xi^{i}q^{-\frac{1}{r + 1}})\,dt.
\end{align*}
So by integration in $t$,
\begin{align*}
R^1_{ii} = \frac{(-1)^r \Xi_r}{(r + 1)^3\lambda}(\xi^{-i}
q^{\frac{1}{r + 1}} + \xi^{i} q^{-\frac{1}{r + 1}}).
\end{align*}

Since we are only interested in the non-equivariant limit $\lambda
\to 0$, by (\ref{canonical}) the only terms remaining in
$\sum_{i=0}^{r} R^1_{ii}\,du_i$ is in the $dt_1$ direction. Thus
\begin{align*}
\sum_{i=0}^{r} R^1_{ii}\,du_i &= \frac{\Xi_r}{(r+1)^3\lambda}
\sum_{i=0}^{r}
(c_i^{-1} + c_i) \frac{\lambda}{a_i}\,dt\\
&= \frac{\Xi_r}{(r+1)^3} \sum_{i=0}^{r} \Big((c_i^{-1} + c_i)
\prod_{j\neq i} a_j \Big) \frac{dt}{1 + (-1)^rq^{-1}}.
\end{align*}

The sum can be computed as (using $a_j = 1 + c_j$ and Lemma
\ref{trivial})
\begin{align*}
&\sum_{i=0}^{r} \Big((c_i^{-1} - 1)
\prod_{j\neq i} a_j + \prod a_j \Big)\\
&= \sum_{i=0}^{r} \Big((c_i^{-1} - 1)(1 - c_i + c_i^2 - \cdots +
(-1)^r c_i^r) + (1 + (-1)^r q^{-1}) \Big)\\
&= (r+1)(-2 + 1 +(-1)^rq^{-1}),
\end{align*}
and we get
\begin{equation} \label{3rd_term}
\sum_{i=0}^{r} R^1_{ii}\,du_i = \frac{\Xi_r}{(r+1)^2} \frac{1 -
(-1)^rq}{1 + (-1)^rq}\,dt.
\end{equation}

So by (\ref{dG}), Proposition \ref{2nd_term}, Proposition
\ref{1st_term} and (\ref{3rd_term}),
$$
dG =\left\{\left[\frac{r}{48} + \frac{\Xi_r}{2(r + 1)^2}\right]
\frac{1 - (-1)^rq}{1 + (-1)^rq} - \frac{(-1)^r(r + 1)^2}{24}
\frac{q}{1 + (-1)^rq}\right\}\, dt.
$$

Thus $dG = (a + b\f)\,dt$ for some $a$ and $b$. In Givental's
formula $a$ should be ignored since $dG$ has no constant terms.
Nevertheless the constant $\Xi_r$ can be explicitly computed:

\begin{lemma} \label{key}
\begin{align*}
\Xi_r = -\frac{(r+2)(r+1)^2r}{24}.
\end{align*}
\end{lemma}

We leave the interesting proof to the readers. With it, a simple
substitution leads to the final result (a redundant constant $-r(r
+ 1)/48$ has been removed):
\begin{theorem} \label{t:g1}
$$
dG = \left[\frac{(-1)^{r+1}(r+1)}{24} \frac{q}{1 - (-1)^{r + 1}
q}\right] d\log q.
$$
\end{theorem}

\end{document}